\documentclass[12pt,leqeq]{article}
\usepackage{amssymb,amsthm,mathrsfs}
\usepackage[utf8]{inputenc}
\usepackage{youngtab}
\usepackage{amsmath}
\usepackage{amscd}
\usepackage{xy}
\xyoption{all}
\usepackage[dvips]{graphicx}
\usepackage{graphpap}
\usepackage{youngtab}
\usepackage{pstricks}
\usepackage{pst-all}
\usepackage{pstcol}
\usepackage{color}
\usepackage[T1]{fontenc}
\usepackage{tgbonum}
\usepackage{tikz}
\usepackage{tikz-cd}
\usetikzlibrary{decorations.pathreplacing}

\usepackage{stmaryrd}
\usepackage{ytableau}
\newtheorem{thm}{Theorem}
\newtheorem{cor}[thm]{Corollary}
\newtheorem{prop}[thm]{Proposition} 
\newtheorem{rem}[thm]{Remark}

\newtheorem{exmp}[thm]{Example}

\date{}

\begin{document}
\setlength{\baselineskip}{16pt}
\title{Combinatorial Approach to the Second Law}
\author{Rafael D\'\i az}
 \maketitle
\begin{abstract}
We study the second law in the context of combinatorial processes, focusing on the mechanisms that give rise to irreversible behavior from an underlying deterministic, invertible, and reversible dynamics.
\end{abstract}

\section{Introduction}

A key feature of the second law is that it inherently operates on two scales—a microscale and a macroscale. Intuitively, the microscale corresponds to a complete description of the system's state (i.e., points in phase space), while the macroscale is the scale accessible to the observer, both operationally and informationally. This duality finds its geometric expression in the introduction of a distinct space, beyond phase space, known as state space, whose points label the equilibrium states of a system.  As discussed below, the second law characterizes possible transitions between equilibria in terms of entropy changes, while presupposing the underlying process of equilibration 
\cite{bu}.\\

Boltzmann, most famously via his $H$-theorem \cite{g0, g}, introduced an extension of the second law that applies to the equilibration process itself, i.e. to transitions between non-equilibrium states\footnote[1]{
In a different usage, the term 'non-equilibrium states' often refers to  non-equilibrium steady states \cite{derri}.} in isolated systems 
(i.e. not interacting with the outside.) This was made possible by introducing a \textbf{macrostate space} whose points label
both equilibrium and non-equilibrium states available at macroscale, and by extending the notion of entropy as counting microstates compatible with a given macrostate.\\

Our aim in this work is to study entropy and the out of equilibrium second law, following the Boltzmann approach, within a combinatorial setting \cite{dz}. 
We specifically focus on the mechanisms that gives rise to irreversible process starting from an deterministic reversible process. The connection between the second law and combinatorics dates back to Boltzmann. In his investigation of equilibrium macrostates, in a system of indistinguishable particles distributed among discrete energy levels, he was led to study the asymptotic behavior of multinomial coefficients using Stirling's approximation. Niven in \cite{ni, ni2} provides a general picture of the use of combinatorial methods along this line. The main advantage of the combinatorial approach is that  rigorous definitions and computations, that may be elusive in the continuous case, become available allowing hypothesis and conjectures to be probed. A natural way to construct combinatorial systems is going through a double process of discretization,
as in Examples \ref{epdt}, \ref{exg}, \ref{qth1}, i.e. discretization at micro and macro scales. This process typically results in stochastic dynamics, but here we focus on deterministic invertible dynamics.  However, from the combinatorial viewpoint, it is natural to postpone taking this limit to the very end, as it has been successfully done in other areas, for example, in $h$-calculus and $q$-calculus \cite{k}. 
\\

Although the combinatorial definitions used in this work are completely self-contained, and require no prior background in physics, for the reader benefit, we briefly explain where these definitions come from, and what are the intended applications. For in depths introductions to this topic, the reader may consult \cite{abn, g, j0, j1}. In thermodynamics entropy $\ S_T \ $ is a function on state space whose points parameterize \textbf{equilibrium states} of the system. The \textbf{second law} states that when an isolated system in equilibrium  $\ X\ $ has a constraint removed  (e.g. a wall is removed,) causing  the system to evolve (in phase space) to a new equilibrium $\ Y, \ $ then entropy increases, i.e. $\ S_T(X)\leq S_T(Y). \ $  The second law only demands that
entropy be defined at the beginning and at the end of this process. 
There are three different definitions of entropy $\ S_T, \ $ which turn out to be equivalent, in the sense that in the examples where this quantities have be computed, they give the same result (in the thermodynamic limit, up to a constant). Jaynes  \cite{j} is a main reference on this topic, however, there seems to be a lack of general formal rigorous statements and proofs of this equivalence in the literature. \\

\textbf{Clausius entropy} $\ S_C \ $ is defined via the differential equation on state space
 \begin{equation}\label{ce}
   dS_C\ = \ \frac{dQ}{T},
 \end{equation}
  where $\ dQ \ $ is the heat received, and $\ T \ $ is the temperature  of the system.  From this equation  $\ S_C(X)\ $ can be obtained, up to a constant, by integration over paths in state space, i.e. integration over reversible paths. The second law follows the inequality\footnote[1]{Here the system receives heat $\ dQ \ $ from a reservoir at temperature $\ T. \ $ Note that the temperature of the system may not be defined as it step away from equilibrium. Clausius  justified  this inequality  
  from Carnot's principle.} 
   \begin{equation}\label{ci2}
   \oint  \frac{dQ}{T} \ \leq \ 0, 
 \end{equation}
  which is an identity for a reversible process, and strict for irrevesible processes.
  \\

\textbf{Gibbs entropy} is based on an explicitly description of equilibrium states as probability distributions on phase space. It is given by 
 \begin{equation}\label{ge}
    S_G(X)\ = \ kH(p_X),
 \end{equation}
where $\ k \ $ is the Boltzmann constant, and $\ H(p_X) \ $ is the Shannon entropy  (\ref{pent}) of the \textbf{equilibrium} distribution $\ p_X,\ $
i.e. $\ p_X\ $ is  the distribution of maximum
Shannon entropy $\ H(p) \ $ among all probability distributions
$\ p \ $ on phase space compatible with the state point $\ X. \ $ The connection between Gibbs and Clausius entropy comes from the fact that, in the intended examples,  $\ p_X \ $ is given explicitly by the Maxwell-Boltzmann distribution, so  $\ S_G(X) \ $ can be computed and it is shown that it satisfies Clausius differential equation (\ref{ce}). The second law states that when a system in equilibrium  $\ X \ $ has a constrain removed, and the system evolve to a new equilibrium $\ Y,\ $ then $\ H(p_X) \leq H(p_Y).\ $ Assuming that evolution preserves Shannon entropy, which  follows from Liouville's theorem, and that $\ p_X\ $  is transported to a probability distribution $\ p'_X\ $ compatible with $\ Y, \ $ then
$\ \ H(p_X) =  H(p'_X)  \leq H(p_Y),\ \ $
because $\ H(p_Y)\ $ is maximum among all distribution compatible with $\ Y.$
\\

\textbf{Boltzmann entropy} is based on a more intuitive description of equilibrium states which are now understood as subsets of phase space.  It
is given by
 \begin{equation}\label{be}
   S_B(X)\ = \ k\mathrm{ln}W \ = \ k\mathrm{ln}\mathrm{vol}(a_X),
 \end{equation}
where $\ k \ $ is the Boltzmann constant, $\ a_X \ $ is a subset of phase space, and $\ \mathrm{ln}\mathrm{vol}(a_X) \ $ is the logarithmic volume of  $\ a_X. \ $ Intuitively, $\ a_X \ $ is the set of phase space points compatible with the system being at state $\ X. \ $ 
The bridge between Boltzmann and Gibbs entropy is the asymptotic equipartition theorem, which tell us, when applicable, that the distribution $\ p_X \ $ is overwhelmingly concentrated 
on a subset $\ a_X\ $ of phase space such that $\ p_X\ $ is, roughly speaking, uniform on $\ a_X,\ $ thus
$\ \   S_G(X)\ = \ kH(p_X)\  \approx  \ k\mathrm{ln}\mathrm{vol}(a_X) \ = \ S_B(X).\ \ $ The second law states that when a 
system in equilibrium $\ X \ $ has constrain removed, and the system evolves
to a new equilibrium $\ Y,\ $ then  $\ \mathrm{vol}(a_X) \leq \mathrm{vol}(a_Y).\ $ If volume is preserved, then $\ \mathrm{vol}(a_X) \leq \mathrm{vol}(a_Y) \ $  since the reproducible evolution maps points in $\ a_X\ $  to points in $\ a_Y.\ $  \\

This work is organized as follows. In Section \ref{mms} we define micro-macro systems, introduce examples, methods of constructions, and study the classes of isomorphisms of such systems. In Section \ref{ifr} we study how reproducibility gives rise to irreversibility. In Section \ref{icg8} we introduce coarse graining and study how it leads to irreversibility. 
In Section \ref{nmsp} we introduced the periodic non-Markovian stochastic process
suitable to address the second law within our combinatorial settings. In Section \ref{ioa} we study entropy production on average. In the final Section \ref{ifa}
we show how irreversibility arises from  attraction.

\section{Micro-Macro Systems}\label{mms}

An invertible combinatorial micro-macro dynamical system  $\ (X,A,m,\alpha) \ $ is given by a finite set of microstates $\ X,\ $ a finite set of macrostates $\ A,\ $ an invertible map $\ \alpha: X \to X, \ $ and a micro to macro surjective map  $\ m:X \to A.\ $ If  $\ m \ $ is  an arbitrary  map,  we can always replace $\ A\ $ by $\ m(X)\ $ to ensure surjectivity. We write $\ i\in a\ $ instead of $\ i\in m^{-1}(a),\ $ i.e. we think of a macrostate 
as a subset of $\ X \ $ whenever convenient.  For $\ a,b \in A\ $ set 
$$ \llbracket a,b\rrbracket \ = \ \{i\in a\ | \ \alpha(i) \in b \} \ \subseteq \ a.$$
Sets $ \ A \ $ and $\ X\ $ come with several additional structures.
Measure $\ |\ |:A \to \mathbb{N}\ $ given by $\ |a|  = |m^{-1}(a)|; \ $ measure $\ |\ |:X \to \mathbb{N}\ $ given by $\ |i|=|m(i)|; \ $ 
probability measure $\ p \ $ on $\ A \ $ given by $\ \displaystyle p_a   =  \frac{|a|}{|X|};\ $ stochastic map $\ [\alpha]: A \to  A \ $ such that $\ [\alpha]_{ab} \in [0,1] \ $ is the probability that the macrostate $\ a \ $ moves to a macrostate $\ b \ $ in a unit of time:
\begin{equation}
[\alpha]_{ab}  \ = \ \mathrm{Pr}[\ \alpha(i) \in b \ | \ i\in a\ ] \ = \ \frac{ |\llbracket a,b\rrbracket |}{|a|}
 \ = \ \frac{ |\alpha (a) \cap b |}{|a|}
 \ = \ \frac{ |a \cap \alpha^{-1}(b)  |}{|a|}.
\label{T}
\end{equation} 
Note that $\  [\alpha]_{ab}\ $ is well-defined for $\ a \ $ and $\ b \ $ arbitrary non-empty subsets of $\ X. \ $  The Boltzmann entropy on macrostates $\ S:A \rightarrow \mathbb{R}_{\geq 0}\ $ is given by $\ S(a) =   \mathrm{ln}|a|. \ $\footnote[1]{We set Boltzmann constant to $\ 1.$}  Again $\ S \ $ is well-defined for arbitrary non-empty subsets of $\  X, \ $ in particular $\ S(X)=\mathrm{ln}(X). \ $ The Boltzmann entropy  on microstates $\ S:X \rightarrow \mathbb{R} \ $ is given by $\ S(i)   =   \mathrm{ln}|m(i)|.\ $ 
The mean Boltzmann entropy $\ S: \mathrm{Prob}(A) \to \mathbb{R} \ $ is given 
on a probability measure $\ q \in \mathrm{Prob}(A)\ $  by
\begin{equation}\label{mse}
S(q) \ = \ \langle S \rangle_q \ = \  \sum_{a\in A}q_aS(a). 
\end{equation}
Shannon entropy $\ H: \mathrm{Prob}(A) \to \mathbb{R} \ $ is
given on $\ q \in \mathrm{Prob}(A)\ $ by
\begin{equation}\label{pent}
H(q)\ = \ -\sum_{a\in A} q_a\mathrm{ln}(q_a).
\end{equation}
For $\ \displaystyle p_a= \frac{|a|}{|X|},\ $ Shannon entropy $\ H(p)\ $ and Boltzmann entropy $\ S(p) \ $ play complementary roles as  $\ H(p)+S(p)=\mathrm{ln}|X|.\ $ In general, Shannon entropy $ \ H(q)\  $ measures the mean
uncertainty in choosing  a macrostate, while Boltzmann entropy $\ S(q) \ $ measures the mean uncertainty in choosing a microstate given that a macrostate has already been chosen. Thus, $\ H(q)   +   S(q)\ $ measures the uncertainty in choosing a microstate according to the coarse-grained probability measure $\ cq \in \mathrm{Prob}(X)\ $ given on $\ i\in X\ $ by 
\begin{equation}\label{cgrain}
 (cq)_i   \  = \ \frac{q_{mi}}{|i|}  \ = \ \frac{q_a}{|a|} \ \ \ \ \mbox{for} \ \ \ \ i \in a.  
\end{equation}
Indeed Shannon entropy $\ H(cq)\ $ is given by
\begin{equation}\label{cq}
 H(cq)   \  = \ -\sum_{a\in A}\sum_{i\in a}
 \frac{q_a}{|a|} \mathrm{ln}\big(\frac{q_a}{|a|}\big) 
 \ = \   H(q)  \ + \  S(q).
\end{equation}

The sets of microstates where Boltzmann entropy
is decreasing, constant, increasing are  given by
$\ D    =   \{i \in X \  |  \ S(\alpha i) < S(i) \}, \ $
$\ C  =   \{ i \in X \  | \  S(\alpha i ) = S(i)  \},\ $
$\ I   =   \{i \in X \  | \  S(\alpha i) > S(i)\}, \ $ respectively. \
The \textbf{decreasing} \ (constant , increasing)\  \textbf{entropy ratio} is given by $\  \displaystyle \frac{|D|}{|X|} \ \  \mbox{(}  
\frac{|C|}{|X|} \ ,  \  \frac{|I|}{|X|}.$)\ \  We have that
\begin{equation}\label{dic=1}
\frac{|D|}{|X|} \ + \ \frac{|C|}{|X|}  \ + \ \frac{|I|}{|X|} \ = \ 1.
\end{equation}

A system $ \ (X,A,m,\alpha) \ $ has an $\varepsilon$-\textbf{dominant equilibrium} if 
\begin{equation}\label{edo}
  \frac{|X^{\mathrm{eq}}|}{|X|}\ \geq \ 1- \varepsilon, 
  \mbox{\ \ \ or \ equivalently \ \ }  \frac{|X^{\mathrm{neq}}|}{|X|} \ \leq  \ \varepsilon,
\end{equation}
where $\ X^{\mathrm{eq}}\ $ is the set of microstates that belong to a macrostate of maximum entropy. If there is a unique $\varepsilon$-dominant macrostate $\ a, \ $  we have that
\begin{equation}\label{edo2}
  \frac{|a|}{|X|}\ \geq \ 1- \varepsilon, 
  \mbox{\ \ \ \ thus \ \ \ \ }  
  \mathrm{ln}(1 -\varepsilon)+\mathrm{ln}|X| \ \leq \ S(a) \ \leq 
  \ \mathrm{ln}|X|.
\end{equation}

\begin{prop}\label{yqt}
{\em Let $\ (X,A,m, \alpha)\ $ be a system with an
$\ \varepsilon$-dominant equilibrium.  Then
it has $\ \varepsilon$-bounded decreasing entropy ratio, $\ \varepsilon$-bounded increasing entropy ratio, and over $\ 1-2\varepsilon\ $ 
constant entropy ratio. 
}
\end{prop}

\begin{proof}[\textbf{Proof}]Follows from $\ \alpha D \subseteq X^{\mathrm{neq}}, \ \ \ I \subseteq X^{\mathrm{neq}}, \ \ $ identity (\ref{dic=1}), \ and the inequalities 
  \begin{equation*}
  \frac{|D|}{|X|}  \ = \  \frac{|\alpha D|}{|X|}
  \ \leq \ \frac{|X^{\mathrm{neq}}|}{|X|} \ \leq  \ \varepsilon 
  \ \ \ \ \ \mbox{and} \ \ \ \ \
  \frac{|I|}{|X|} 
  \ \leq \ \frac{|X^{\mathrm{neq}}|}{|X|} \ \leq  \ \varepsilon.
\end{equation*}
\end{proof}

There are several interesting constructions for building new systems out of given systems  $ \ (X,A,m,\alpha), \ (Y,B, t,\beta).\ $ \\

\noindent 1) \textbf{Disjoint Union}  $ \ (X,A,m,\alpha) \sqcup  (Y,B, t,\beta)\ = \ (\ X \sqcup Y,\ A \sqcup B,\ m \sqcup t,\ \alpha \sqcup \beta\ ).\ $ 
We have that $\ S(p_{A \sqcup B})\ = \ \frac{|X|}{|X|+|Y|}S(p_A) \ + \ \frac{|Y|}{|X|+|Y|}S(p_B).$ \\

\noindent 2) \textbf{Product} $ \ (X,A,m,\alpha)\times (Y,B, t,\beta) \ = \ 
(\ X\times Y, \ A \times B,\ m \times t, \ \alpha \times \beta \ ). \ $ We have that $\ S(p_{A \times B}) \ = \ S(p_{A}) \ + \ S(p_{B}).$\\

\noindent 3) \textbf{Inverse System} $ \ (X,A,m,\alpha^{-1}). \ $ 
We have $\ \  D_{\alpha^{-1}}=\alpha I, \ $ $\ C_{\alpha^{-1}} =\alpha C,\ \ $
$ I_{\alpha^{-1}} = \alpha  D. \ \ $  Therefore
$\ \  |D_{\alpha^{-1}}|=|I|,\ \ \  
|C_{\alpha^{-1}}|=| C|,\ \ \ |I_{\alpha^{-1}}|=|D|.$ \\

\noindent 4) \textbf{Coarsening} $ \ (X,C,cm,\alpha), \ $ where
$\ c:A \to C \ $ is a surjective map. We have that
$\ S(p_C) \geq S(p_A) .\ $ If  $\ a \in A\ $ is 
$\ \varepsilon$-dominant equilibrium, the $\ ca \in C\ $
is also an $\ \varepsilon$-dominant equilibrium.  \\

\noindent 5) \textbf{Reunion} $ \ (X,A,m,\alpha) \cup   (Y,B, t,\beta) \ = \ 
(\ X \sqcup Y,\ A \cup B,\  m \cup t,\ \alpha \sqcup \beta\ ).\ $ 
Reunion can be obtain by applying first disjoint union, and
then coarsening with the canonical map $\ c: A \sqcup B \to A \cup B. \ $ 
Thus 
$\ S(p_{A \cup B}) \ \geq \ \frac{|X|}{|X|+|Y|}S(p_A) \ + \ \frac{|Y|}{|X|+|Y|}S(p_B). $ \\

\noindent 6) \textbf{Restriction} $ \ (U,\ mU, \ m,\ \alpha_U ) \ $ of
$\ (X,A,m,\alpha),\ $ where
$\ U \subseteq X \ $ and $\ \alpha_U \ $ is such that  $\ \alpha_U$-cycles
are the restriction of $\ \alpha$-cycles to $\ U.\ $ If $\ a\in mU, \ $
then $$\ S_U(a) \ = \ \mathrm{ln}|a\cap U| \ \leq \ \mathrm{ln}|a| \
= \ S(a). $$ 

\noindent 7) \textbf{Extensive Joint} 
\ For $\ A,  B \ \subseteq \ \mathbb{R}\ $  set
$$ (X,A,m,\alpha) \vee  (Y,B, t,\beta) \ = \ 
(\ X\times Y,\ A + B, \ m\pi_1 + t\pi_2,\ \alpha \times \beta \ ).$$
We have  $\ \ \displaystyle p_{A+B}(c) \ = \ \sum_{a+b=c}p_A(a)p_B(b) \ \ $
where $\ \ a\in A, \ b\in B.$ \\

\noindent 8) \textbf{Iteration} $ \ (X,m^{l}X,m^l,\alpha) \ \ $ where 
$\ \ m^l(i)\ = \ ( mi,\ m\alpha i,\ \dots, \ m\alpha^{l-1}i ) \ \in\ A^{l}.\ \ $ Note that $\ \ (\pi_1,\dots, \pi_k)m^l=m^k \ \ $ for $\ \ l\geq k, \ \ $ thus
$\ \ S(p_{m^{k}X})\geq S(p_{m^{l}X}).$\\

\noindent 9) \textbf{Zones} $ \ (X,Z,|m|,\alpha) \ $ where
$\ Z = \{ |a| \ | \ a \in A\}\ $ and $\ |m|(a)=|m(a)|. \ $ 
We have that $\ \ S(p_Z) \geq S(p_A).$\\

For $\ n\in \mathbb{N}\ $ we use the notation
$\ [n]=\{1,...,n\}.$

\begin{exmp}\label{epdt}
{\em Consider system $ \ (X, A, m, \alpha). \ $  Let
$\ \{P_c\}_{c\in C} \ $ be a finite semi-open partition\footnote[1]{A subset $\ P \ $ of a topological space is \textbf{semi-open} if 
$\ \mathrm{cl}(P)  = \mathrm{cl}(\mathrm{Int}(P)),\ $ i.e. it's closure is equal to the closure of it's interior.   A partition is semi-open if its blocks  are semi-open.} of the space $\ \mathrm{Prob}(A)\subseteq \mathbb{R}^{A},\ $ and let 
$\ \pi:\mathrm{Prob}(A) \to C \ $ be the canonical map. For 
$\ N \geq 1 \ $ the micro-macro empirical map $\ m_N: X^N \to \mathrm{Prob}(A)\ $ sends
$\ i\in X^{N} \ $ to the probability distribution
$\ m_N(i) \in \mathrm{Prob}(A) \ $ given by
$$m_N(i) \ = \ E_Nm^{\times N}(i)\ = \ \frac{1}{N}\sum_{s=1}^{N}\delta_{m(i_s)} \ , \ \ \ 
\ \mbox{that is} \ \ \ \ m_N(i)_a \ = \ \frac{|\{s\in [N] \ | \ m(i_s) =a \}|}{N}.$$
We obtain the micro-macro dynamical system
$\ (X^N,  C, \pi m_N, \alpha^{\times N}) .$
}
\end{exmp}

\begin{exmp}\label{exg}
{\em Let finite group $\ G \ $ act on a finite set $\ X, \ $ and let $\ V\subseteq G \ $ be invariant under inversion $\ V^{-1}=V. \ $ 
We have bijection $\ \alpha: X\times V \to X\times V \ $ given by $\ \alpha(i,v)=(iv,v), \ $ with inverse map $\ \alpha^{-1}(i,v)=(iv^{-1},v). \ $
Let  $\ \{P_c\}_{c\in C} \ $ be a finite semi-open partition of $\ \mathrm{Prob}(X\times V)\ $ and  $\ \pi:\mathrm{Prob}(X\times V) \to C \ $ be the canonical map.
For $\ N\geq 1\ $ let $\  E_N : (X\times V )^N \to \mathrm{Prob}(X\times V) \ $ be the empirical map,   we get system 
$\ ((X\times V )^N,  C,  \pi E_N, \alpha^{\times N}  ).  \ \  $
Similarly, let $\ \{P_c\}_{c\in C} \ $ be a finite semi-open partition of $\ \mathrm{Prob}(X),\ $
we get system $\ ( (X\times V )^N,  C,  \pi E^1_N, \alpha^{\times N} ), \ $
where  $\ \pi E^1_N \ $ is the composition map
$\  (X\times V )^N \to X^N \to \mathrm{Prob}(X) \to C . \ $ }
\end{exmp}

For $\ n\in \mathbb{N}_{\geq 2}\ $ we set 
$\ \mathbb{Z}_{n}= \mathbb{Z}/n\mathbb{Z}. $

\begin{exmp}\label{qth1}
{\em Following Example \ref{exg}, let 
$\ X=G=\mathbb{Z}_{2d}\ $ and $\ V=\mathbb{Z}_{2d}^{\ast}=\mathbb{Z}_{2d}\setminus \{0\}.\ $ 
Consider the map $\ c:\mathbb{Z}_{2d} \to \mathbb{Z}_{2}\ $
sending $\ \{0,\dots,d-1 \}\ $ to $\ 0, \ $ and
$\ \{n,\dots,2d-1 \}\ $ to $\ 1. \ $ 
Fix a finite semi-open partition $\ \{P_c\}_{c\in C} \ $ of $\ \mathrm{Prob}(\mathbb{Z}_{2}). \ $ 
Let $\ t_N \ $ be the map given by the composition
$$  (\mathbb{Z}_{2d}\times \mathbb{Z}_{2d}^{\ast} )^N \to \mathbb{Z}_{2d}^N \to
\mathbb{Z}_{2}^N \to \mathrm{Prob}(\mathbb{Z}_{2}) \to C  $$
where the second map is $\ c^{\times N} , \ $ the third is the empirical distribution,  and the fourth is the canonical map.
We get system $$ ( (\mathbb{Z}_{2d}\times \mathbb{Z}_{2d}^{\ast} )^N, 
C,  t_N, \alpha^{\times N} ). \ $$
}
\end{exmp}

A family $\ (X_n, A, m_n, \alpha_n)\ $ of systems has  \textbf{large deviations with rate function} $\ I:A \to \mathbb{R}_{\geq 0}\ $ if setting 
$\ a_n = m_n^{-1}a \ $ we have that 
\begin{equation}\label{ld}
\lim_{n \to \infty} \frac{1}{n}\mathrm{ln}[p_n(a)]
\ = \ 
\lim_{n \to \infty} \frac{1}{n}\mathrm{ln}\bigg(\frac{|a_n|}{|X_n|}\bigg) 
\ = \ -I(a) \ \ \ \ \mbox{for} \ \ \ a\in A.
\end{equation}

\

\begin{thm}\label{yqt}
{\em Fix $\ \varepsilon >0. \ $ Let systems $\ (X_n, A, m_n, \alpha_n)\ $ have  large deviations with injective rate function $\ I:A \to \mathbb{R}_{\geq 0}.$

\noindent a) $(X_n, A, m_n, \alpha_n)\ $ has a $\ \varepsilon$-dominant equilibrium for $\ n \ $ large enough. 

\noindent b) Set $\ Da_n=\{i\in a_n \ | \ S(\alpha_n i) < S(i)\}, \ $ then
$\ \ \displaystyle \frac{|Da_n|}{|a_n|} \ \leq \ \varepsilon \ \ $ for $\ \ n \ \ $ large enough.

}
\end{thm}

\begin{proof}[\textbf{Proof}] Fix $\ \delta > 0 \ $ small enough such that
$\ I(b)>I(a)\ $ implies $\ I(b)>I(a)+2\delta \ $ for 
$\ a,b \in A. \ $
Let $\ I(b)>I(a)\ $ and $\ n\to \infty ,\ $  then by (\ref{ld}) we have that
\begin{equation}\label{ycsde}
\frac{|b_n|}{|a_n|}  \ = \  
\frac{\frac{|b_n|}{|X_n|}}{\frac{|a_n|}{|X_n|}} \ = \ 
e^{n\big( \frac{1}{n}\mathrm{ln}(\frac{|b_n|}{|X_n|} - 
\frac{1}{n}\mathrm{ln}(\frac{|a_n|}{|X_n|}  \big)}
\ \leq \ e^{n\big(-I(b) +I(a) + 2\delta  \big)}
\ \to \ 0
\end{equation}

\noindent a) $I\ $ is injective, so it achieves its minimum at a unique macrostate
$\ a^{\mathrm{eq}}.\ $ Then
$$\frac{|a_n^{\mathrm{eq}}|}{|X_n|} \ = \  \frac{1}{\sum_{a\in A}\frac{|a_n|}{|a_n^{\mathrm{eq}}|}}
\ = \  \frac{1}{1+\sum_{a\neq a^{\mathrm{eq}} }\frac{|a_n|}{|a_n^{eq}|}} \ \to \ 1 \ \ \ \ \mbox{as} \ \ \ \ n \to \infty \ \ \  \mbox{by} \ \ \  (\ref{ycsde}).$$

\noindent b) As $\ n \to \infty\ $  by (\ref{ycsde}) we have that
\begin{equation}\label{y22}
\frac{|Da_n|}{|a_n|} \ \leq \  
\frac{\sum_{I(b)>I(a)}|b_n|}{|a_n|} \ = \ 
\sum_{I(b)>I(a)}\frac{|b_n|}{|a_n|}
 \ \to \ 0
\end{equation}
\end{proof}

Recall that for a finite set $\ C\ $ the Kullback-Leibler divergency 
\begin{equation}\label{KL}
\ D(\ \| \ ):\mathrm{Prob}(C) \times \mathrm{Prob}(C) \rightarrow [0,\infty] 
\ \ \ \ \mbox{is given by} \ \ \ \ D(p \| q )=\sum_{c\in C} p_c\mathrm{ln}\big(\frac{p_c}{q_c}\big).
\end{equation}

\begin{thm}\label{ct3}
{\em Fix small $\ \varepsilon > 0, \ $
and $\ k\geq 1.\ $ Systems $\ ( X^N,  [k], \pi m_N,\alpha^{\times N} ) \ $ 
constructed as in Example \ref{epdt} 
   have large deviations with injective rate function $$\ I(l)=(l-1)\varepsilon, \ \ \ \ \ \mbox{ where }$$
   $$  P_l\ = \ \{q\in \mathrm{Prob}(A) \ | \  
   (l-1)\varepsilon \leq D(q\| p)< l\varepsilon\}\ \ \
    \mbox{for} \ \ \ 1\leq l \leq k-1, $$ 
    $$\ P_k \ = \ \{q\in \mathrm{Prob}(A) \ | \  
   (k-1)\varepsilon \leq D(q\| p)\}. $$
}
\end{thm}

\begin{proof} First note that
\begin{equation}\label{noma}
\frac{|(\pi m_N)^{-1}(l)|}{|X^N|} \ \ = \ \ 
\sum_{(i_1,\dots,i_N)\in (m^{\times N})^{-1}E_N^{-1}\pi^{-1}(l) }\frac{1}{|X|^N}
\ \ =   
\end{equation}
$$\sum_{(a_1,\dots,a_N)\in E_N^{-1}P_l }
\frac{|a_1|\cdots |a_N|}{|X|^N} \ \ = \ \ 
\sum_{E_N(a_1,\dots,a_N)\in P_l }
p_{a_1} \cdots p_{a_N}\ = \ p^{\times N}\big(a\in A^{N} \ \big| \ E_N(a) \in P_l \big).$$

Next we apply Sanov's theorem \cite{dz, e}:

\begin{equation}\label{noma2}
\lim_{n \to \infty} \frac{1}{n}\mathrm{ln}\frac{|(\pi m_N)^{-1}(l)|}{|X^N|} \ \ = \ \ 
\lim_{n \to \infty} \frac{1}{n}\mathrm{ln} \ p^{\times N}\big(a\in A^{N} \ \big| \ E_N(a) \in P_l \big)
\ \ =   
\end{equation}
$$-\inf_{q\in P_l}D(q||p) \ = \ -(l-1)\varepsilon \ = -I(l).$$
We have used that the blocks $\ P_l\ $ are semi-open:
$\ p \ $ is strictly positive, so $\ D(q||p) \ $ is a bounded continuous function on $\ q\in \mathrm{Prob}(A), \ $ and therefore  for $\ 1\leq l \leq k-1\ $ we have that
$$\mathrm{cl}(\mathrm{int}(P_l)) \ = \ \mathrm{cl}(\{q\in \mathrm{Prob}(A) \ | \  
   (l-1)\varepsilon < D(q\| p)< l\varepsilon\}) \ = $$
   $$ \{q\in \mathrm{Prob}(A) \ | \  
   (l-1)\varepsilon \leq D(q\| p) \leq l\varepsilon\} \ = \
   \mathrm{cl}(P_l). $$
Similarly $\ \mathrm{cl}(\mathrm{int}(P_k)) =  
   \mathrm{cl}(P_k). $

\end{proof}

\begin{cor}\label{ct79}
{\em Fix small $\ \varepsilon > 0 \ $
 and $\ k\geq 1.\ $ \\

\noindent a) Systems $ \ ((X\times V )^N,  [k], \pi E_N,\alpha^{\times N} )\  $ 
constructed as in Example \ref{exg} 
   have large deviations with injective rate function $\ I(l)=(l-1)\varepsilon, \ \ $
   where  $$\ \ P_l\ = \ \{q\in \mathrm{Prob}(X\times V) \ | \  
   (l-1)\varepsilon \leq D(q\| u)< l\varepsilon\}\ \ 
    \mbox{for} \ \  1\leq l \leq k-1, \ \  $$  
    $$\ \  P_k \ =\ \{q\in \mathrm{Prob}(X\times V) \ | \  
   (k-1)\varepsilon \leq D(q\| u)\}.$$

\noindent b) Systems $\ ((X\times V )^N, [k],  \pi E^1_N,\alpha^{\times N} ) \ $ 
constructed as in Example \ref{exg} 
   have large deviations with injective rate function $\ I(l)=(l-1)\varepsilon, \ \ $
   where  $$\ \ P_l \ = \ \{q\in \mathrm{Prob}(X) \ | \  
   (l-1)\varepsilon \leq D(q\| u) < l\varepsilon\}\ \ 
    \mbox{for} \ \  1\leq l \leq k-1, \ \  $$  
    $$\ \  P_k \ = \ \{q\in \mathrm{Prob}(X) \ | \  
   (k-1)\varepsilon \leq D(q\| u)\}.$$

\noindent c) Systems 
  $\  ((\mathbb{Z}_{2n}\times \mathbb{Z}_{2n}^{\ast} )^N, [k],  t_N, \alpha^{\times N}) \ $ constructed as in Example \ref{qth1} 
   have large deviations with injective rate function $\ I(l)=(l-1)\varepsilon, \ \ $
   where  $$\ \ P_l \ = \ \{q\in \mathrm{Prob}(\mathbb{Z}_2) \ | \  
   (l-1)\varepsilon \leq D(q\| u) < l\varepsilon\}\ \ 
    \mbox{for} \ \  1\leq l \leq k-1, \ \  $$  
    $$\ \  P_k \ = \ \{q\in \mathrm{Prob}(\mathbb{Z}_2) \ | \  
   (k-1)\varepsilon \leq D(q\| u)\}.$$
}
\end{cor}

\begin{proof}
The displayed systems arise from the constructions of Example \ref{epdt} and
Theorem \ref{ct3} as follows:
$$\mbox{Systems}  \ \ ((X\times V )^N,  [k], \pi E_N,\alpha^{N} )\ \ \mbox{arise \ from} \ \ (X\times V, X\times V, 1, \alpha).  $$
$$\mbox{Systems}  \ \ ((X\times V )^N, [k],  \pi E^1_N,\alpha) \ \  \mbox{arise \ from} \ \ (X\times V, X, \pi_1, \alpha).  $$
$$\mbox{Systems}  \ \ ((\mathbb{Z}_{2n}\times \mathbb{Z}_{2n}^{\ast} )^N, [k],  t_N, \alpha^{\times N})\ \ \mbox{arise \ from} \ \ (\mathbb{Z}_{2n}\times \mathbb{Z}_{2n}^{\ast}, \mathbb{Z}_{2},  c\pi_1, \alpha). $$
In the three cases the distribution $\ p \ $ on macrostates is uniform.
\end{proof}

Note that there are  $\ n!k!S(n,k) \ $ micro-macro dynamical systems with $\ n \ $ microstates and  $\ k\ $ macrostates, where $\ S(n,k) \ $ are the Stirling numbers of the second kind. Indeed $\ n!\ $ counts bijections  $\ [n]\to [n], \ $  and $\ k!S(n,k)\ $ counts surjective maps $\ [n]\twoheadrightarrow [k]. \ $ \\

Two systems $\ (X_1,A_1,m_1,\alpha_1) \ $ and $\ (X_2,A_2,m_2,\alpha_2) \ $ are isomorphic, and thus considered as essentially the same, if there is a bijection $\ \beta: X_1 \to X_2 \ $ that transports both dynamics and macrostates, i.e. 
$\ \alpha_2= \beta \alpha_1\beta^{-1} \ $ and $\ \beta \ $ induces a (necessarily unique) map $\ \widehat{\beta}: A_1 \to A_2 \ $ such that $\ \widehat{\beta}f_1 = f_2\beta. \ $ System $\ (X,A,m,\alpha) \ $
is isomorphic to system $\ (X,\pi,c,\alpha) \ $ via the identity map $\ \beta=1, \ $ where $\pi \in \mathrm{Par}(X)$ is the partition of $\ X \ $
given by  $\ \pi=\{m^{-1}(a)\ | \ a\in A\} \ $  and $\ c:X \to \pi \ $ is the canonical map. The map $\ \widehat{\beta}: A \to \pi \ $
is given by $\ \widehat{\beta}(a) = m^{-1}(a). \ $
Two systems $\ (X,\pi_1,c,\alpha_1) \ $ and 
$\ (X,\pi_2,c,\alpha_2) \ $ are isomorphic if there is a permutation
$\ \beta: X \to X \ $ such that  $\ \alpha_2= \beta \alpha_1\beta^{-1} \ \  $ and $\ \pi_2=\beta\pi_1. \ $

\begin{exmp}\label{e1o}
{\em Let $\ d:\mathbb{N}_{\geq 1} \to \mathbb{N}_{\geq 1} \  $ 
be the map given by
$$ d_n  =   \underset{\pi \in \mathrm{Par}[n], \ \alpha \in \mathrm{S}_n}{\mathrm{max}}\  |D([n],\pi, \alpha )|, \  $$
i.e. $\ d_n \ $ measures the size of the largest possible decreasing entropy set
for a system with $\ n \ $ microstates. 
Let  $\ l=(l_1, \dots, l_n)\ $ be a numerical partition of $\ n, \ $ i.e. $\ l_k \geq 0,\ $ and $\sum_{k=1}^{n} kl_k =n. \ $ Thus
$\ l \ $ represents partitions of $\ [n] \ $ with $l_k$ blocks of
size $\ k. \ $ Consider system $\ ([n],l, \alpha) \ $ defined diagrammatically; \
below we display the diagram for $\ n=25,\ $ and partition with $\ l_1=l_2=l_4=l_6=1, \ $
$l_3=4,$ and  $ \ l_i=0\ $ otherwise. Blocks are numbered 
from $\ 1 \ $ to $\ 7, \ $ and $\ \alpha \ $ is given by the cyclic order that starts at the top of first column and follow the arrows.
The arrow  $\ \nearrow \ $ means moving to the top of the next column.  \\

$\ytableausetup{centertableaux, boxsize=1cm}
\begin{ytableau}
8 \downarrow   & 8\downarrow  & 8\downarrow  & 8\downarrow  & 8 \downarrow  & 
8 \downarrow  \\
7 \downarrow  & 7\downarrow  & 7\downarrow   &  7\downarrow  \\
3 \downarrow  & 3\downarrow  & 3\nearrow  & 4\nearrow  & 4\nearrow  &
4 \rightarrow  & 5 \rightarrow  & 5 \rightarrow  & 5 \rightarrow 
& 6 \rightarrow  & 6 \rightarrow  & 6 \nearrow\\
2 \downarrow  & 2 \nearrow   \\
1 \nearrow 
\end{ytableau}$

\

\

Permutation $\ \alpha\ $ has the largest possible entropy decreasing set,
for this partition, with $\ |D|\ = \ 13 \  = \ 25 - 12\ = \ n- 3l_3.\ $ Indeed \cite{dz}, this construction always results in a permutation of maximum
decreasing entropy set for $\ l, \ $  thus 
$$ \underset{\alpha \in \mathrm{S}_n}{\mathrm{max}}\ |D([n],l, \alpha )| \ =  \  n -  \underset{1 \leq k \leq n}{\mathrm{max}}\ kl_k, \ \ \ \ \ \ \ 
\mbox{and}$$
$$\ \displaystyle d_n   =
  n \ - \ \underset{l \ \vdash \ n}{\mathrm{min}}\
\underset{1 \leq k \leq n}{\mathrm{max}}\ kl_k, \ \ $$ where $\ l \ $ runs over the numerical partitions of $\ n. \ $ 
}
\end{exmp}

\begin{thm}\label{tnk}
{\em  The number of isomorphism classes of micro-macro dynamical systems with $\ n\geq 1 \ $ microstates is given by
\begin{equation}\label{nmds}
\sum_{k=1}^{n}\sum_{ l,\pi }
 \prod_{c\in \pi}\sigma(c) 
\end{equation}
 $$\mbox{where}\ \ \ l\in \mathrm{par}_k(n), \ \ \ \pi \in \mathrm{Par}[k],
\ \ \ g_c=\mathrm{gcd}\{l_i|i\in c\} \ \ \ \mbox{and} 
\ \ \   \sigma(c) =\sum_{d|g_c}d^{|c|-1}.$$
}
\end{thm}

\begin{proof}[\textbf{Proof}]
We want to compute the cardinality of the quotient set
$\ (S_n\times \mathrm{Par}[n])/S_n \ $ where $\ S_n \ $ acts on itself by conjugation, and on partitions $\mathrm{Par}[n]$ by direct image. 
Burnside lemma implies that
\begin{equation}\label{b1}
\big|(S_n\times \mathrm{Par}[n])/S_n\big|  \ = \  
\frac{1}{n!}\sum_{\alpha\in S_n}\big|(S_n\times \mathrm{Par}[n])^{\alpha}\big|
\ = \ \frac{1}{n!}\sum_{\alpha\in S_n}|S_n^{\alpha}|| \mathrm{Par}[n]^{\alpha}|. 
\end{equation}
The number $\ |S_n^{\alpha}|\ $ of permutations commuting with $\ \alpha, \ $ and the number $\ | \mathrm{Par}[n]^{\alpha}| \ $ of partitions invariant under $\ \alpha, \ $ are completely determined by  cycle structure of $\ \alpha. \ $  Assume $\ \alpha \ $ has $\ k \ $ cycles and let $\ l=(l_1,\dots,l_k) \in \mathrm{par}(n)\ $ be the numerical partition of $\ n\ $ obtained by writing in decreasing order the  lengths of $\ \alpha$-cycles. Let $\ c(l) \ $ be the number of permutations conjugate to $\ \alpha. \ $ 
By the orbit stabilizer theorem there are $\ \displaystyle \frac{n!}{c(l)} \ $  permutations commuting with $\ \alpha. \ $ Therefore
\begin{equation}\label{b2}
\frac{1}{n!}\sum_{\alpha\in S_n}|S_n^{\alpha}|| \mathrm{Par}[n]^{\alpha}|
\ = \ \frac{1}{n!}\sum_{l\in \mathrm{par}(n)}c(l)\frac{n!}{c(l)} |\mathrm{Par}[n]^l| \ = \  
\sum_{l\in \mathrm{par}(n)}|\mathrm{Par}[n]^l|.
\end{equation}
Combining $(\ref{b1})$ and $(\ref{b2})$ it remains to compute 
$\ |\mathrm{Par}[n]^l|.\ $ We fix an origin on each cycle of $\ \alpha, \ $ so points on a cycle of length $\ s \ $ are numbered $\ 0,\cdots,s-1. \ $
The final counting however does not depend on this choice.
Now a partition $\ \pi \in \mathrm{Par}[n] \ $ is 
$\ \alpha$-invariant if and only if $\ \alpha \ $ descends to a bijection $\ a \to \alpha(a)\ $ on macrostates. The key observation is that if $\ a_1 \xrightarrow{\alpha} \cdots \xrightarrow{\alpha}  a_d \ $ is an $\ \alpha$-cycle on macrostates, then $\ \coprod_{1}^da_i \ $ is the union of microstates in a set $\ c\subseteq [k]\ $ of  $\ \alpha$-cycles in $\ X\ $ whose lengths are multiple of $\ d,\ $ and an $\ \alpha$-cycle in $\ c \ $ of length $\ s\ $ goes
trough the blocks $\ a_t \ $ respecting the cyclic ordering, passing
$\ \displaystyle \frac{s}{d}\ $ times trough each block $\ a_t. \ $  Let $\ i \ $ be the smallest index in $\ c. \ $
Without loss of generality we may assume that the first element of the cycle $\ i\ $ lies in $\ a_1. \ $ Then for each $\ j \neq i\ $ in $\ c \ $
there is unique $\ 0\leq n_j < d\ $ such that the point $\ n_j \ $ of the cycle 
$\ j \ $ lies in $\ a_1. \ $ Identity (\ref{nmds}) follows since for each
$\ \alpha \in S_n\  $ with cycle structure $\ (l_1,\dots,l_k)\in \mathrm{par}_k(n) \ $
 each $\ \alpha$-invariant partition can be constructed as:
\begin{itemize}
  \item  Choose a partition $\ \pi \in \mathrm{Par}[k] \ $ on the 
   $\ \alpha $-cycles. 
  \item For each block $\ c \in \pi\ $ choose a divisor $\ d|g_c, \ $ 
  where $\ g_c \ $ is the greatest common divisor 
   of the lengths of $\ \alpha$-cycles in $\ c.$
  \item For each block $\ c\in \pi\ $ choose a function $\ c\setminus i \to \{0,\cdots,d-1\} , \ $ where $\ i \ $ is the smallest element of 
      $\ c.$
\end{itemize}
\end{proof}

\section{Irreversibility From Reproducibility}\label{ifr}

Consider system $\  (X,A,m,\alpha). \ $  Following
Jaynes \cite{j} we say that a transition $ \ a \to b \ $ between macrostates is \textbf{reproducible} if 
$\ \alpha a \subseteq b , \ $ i.e. $\ \alpha i \in b \ $ for
$\ i \in a. \ $ A reproducible transition $ \ a \to b \ $ increases Boltzmann entropy, indeed 
\begin{equation}\label{rep}
S(a) \ = \ \mathrm{ln}|a| \ = \  \mathrm{ln}|\alpha a|  \ \leq \  \mathrm{ln}|b| \ = \ S(b)
\end{equation}
Let $ \ (A, \to) \ $ be the directed graph of \textbf{reproducible transitions} on macrostates. \\

\begin{thm}
{\em For any system $\  (X,A,m,\alpha) \ $ its graph $ \ (A, \to) \ $ of reproducible transitions consists of a disjoint union of cycles and rooted trees. Entropy is constant on each cycle. Entropy increases on paths to the roots.
}
\end{thm}

\begin{proof}[\textbf{Proof}]
A macrostate $ \ a \ $ is a sink ( no transitions out of $ \ a \ $)  if and only if $\ \alpha \ $ sends some microstates in $\ a \ $ to different macrostates.   In $ \ (A, \to) \ $ there are no bifurcations $\ a \to b, \ a \to c,\ $ with $ \ b \neq c, \ $ and in a merge $\ a \to c, \ b\to c, \ $ with $ \ a \neq b, \ $ we have $\ S(a)<S(c) \ $ and  $\ S(b)<S(c). \ $ Since $\ A \ $ is a finite set, each chain $\ a_1 \to \cdots \to a_d \to \cdots \ $ either stops in a sink, or returns to itself, and in the latter case it must to be cycle. Indeed, suppose the first return happens after $\ d$-steps. If $\ d=1\ $ we get cycle $\ a_1 \to a_1. \ $ For$\ d\geq 2,\ $ a cycle $\ a_i \to \cdots \to a_d \to a_i, \ $ with $\ 2 \leq i \leq d, \ $ is not possible  for we would have $\ S(a_d)=S(a_i)\ $
and  $\ S(a_d)<S(a_i)\ $ since there is merger $\ a_{i-1} \to a_i, \ a_{d} \to a_i.\ $ Thus for $\ d\geq 2\ $ we get cycle $\ a_1 \to \cdots \to a_d \to a_1 \ $ with constant entropy. \ \ Consider a sink $ \ r \ $ and let $\ T_r \subseteq A \ $ be the set of macrostates $\ a \ $ such that there is a (\ necessarily unique, otherwise there would be bifurcations\ ) path from $\ a \ $ to $\ r.\ $ The induced subgraph $\ (T_r, \to) \subseteq (A, \to) \ $ is a tree with root $\ r.\ $ The cycles together with the trees $\ T_r \ $ form a partition of $\ A. \ $
\end{proof}

From the reproducible viewpoint the system splits into a cyclic part, and an irreversible part representing a process where at each stage  a macrostate moves to a higher entropy macrostate, until it reaches a root macrostate where there is no further reproducible dynamics.

\begin{prop}
{\em Let $ \ (A, \to) \ $ be the graph of
reproducible transitions of  $\  (X,A,m,\alpha), \ $
and $ \ (A, \rightsquigarrow) \ $ be the graph of reproducible transitions 
of $\  (X,A,m,\alpha^{-1}). \ $ \\

\noindent a) Microstates on a cycle $\ a_1 \to \cdots \to a_d  \ $ in $ \ (A, \to) \ $ belong to a union of $\ \alpha$-cycles on $\ X \ $ whose lengths are multiple of $\ d. \ $ Macrostate $ \ v \in A \ $  is a leaf in $\  (A, \to) \ $ if and only if $\ \alpha^{-1}i \ $ lies in a root macrostate for all $\ i \in v.\  $ Cycles in $ \ (A, \rightsquigarrow) \ $ are reversed cycles in $ \ (A, \to). \ $\\

\noindent  b)  There is arrow  $ \ a \rightsquigarrow b \ $ if and only if
 $\ a \subseteq \alpha b.\ $ 
Given $ \ a \rightsquigarrow b, \ $ then either $\ S(a)<S(b)\ $ and $\ b \ $ is a root  of $ \ (A,\to), \ $
 or $\ S(a)=S(b)\ $ and  $\ b\to a.\ $ \\

\noindent  c) A tree $\ T \ $ in $ \ (A,\rightsquigarrow) \ $ is either the reverse of an inner constant entropy path in $ \ (A, \to), \ $
or it splits in a unique fashion into a tree $\ T^r \ $ whose vertices are roots of $ \ (A, \to), \ $ the reverse of an exit path of constant entropy in $ \ (A, \to) \ $ ending at the root of $\ T^r, \ $ the reverse of  entrance paths of constant entropy in $ \ (A, \to) \ $
beginning at a leave included in the $\ \alpha$-image of a vertex in $\ T^r. \ $ 
} 
\end{prop}

\begin{proof}[\textbf{Proof}]
\noindent a) The first statement follows from the proof of Theorem \ref{tnk}. 
Let $\ v \ $ be a leaf of $ \ (A, \to), \ $ if $\ j\in a\ $ and  $\ \alpha j  \in v, \ $  then $\ a \ $  is not a cyclic macrostate. If $\ a \ $ were a non-root macrostate, then  $\ \alpha a \subseteq v, \ $ i.e. we would have link
$\ a \to v, \ $ and $\ v \ $ would not be a leaf.  If $\ \alpha^{-1}i \ $ lies in a root macrostate for all $\ i \in v,\  $ then $v$ is not a cyclic macrostate, and there is no arrow $\ a \to v,\ $ because $\ a \ $ would have to be a root. 
The third statement is clear since $\ b=\alpha a\ $ if and only if $\ a=\alpha^{-1} b.\ $ \\

\noindent  b)  Note that $ a \rightsquigarrow b\ $ implies that $\ |a|\leq |b|,\ $ and there is $\ i\in b \ $ such that $\ \alpha i \in a. \ $ If $\ b \to a \ $ then $\ |b|\leq |a|, \ $ and $\ S(a)=S(b). \ $  Otherwise, there is $\ j\in b \ $ such that $\ \alpha j \notin a, \ $
therefore $\ b \ $ is a root for $\ \to, \ $ and 
$\alpha^{-1}a$ is a proper subset of $\ b, \ $ so $\ S(a)<S(b).  $\\

\noindent c) An inner constant entropy path is a maximal
path $\ a_1 \to \cdots \to a_k \ $ between non-roots non-leaves vertices
in $ \ (A, \to) \ $ such that $\ S(a_1) = \cdots = S(a_k). \ $  The entrance path associated to a leave $\ v\ $ is the maximal path of constant entropy in $ \ (A, \to) \ $ starting at $ \ v. \ $ Similarly, the exit path associated to a root $\ r\ $ is the maximal path of constant entropy in $ \ (A, \to) \ $ ending at $ \ r. \ $ The result follows by considering the four different option for links in
$\ T. \ $ A link $\ r \rightsquigarrow s \ $ between roots in  $ \ (A, \to) \ $ is a link in $\ T^r. \ $ A link $\ r \rightsquigarrow s \ $ with $\ r\ $ root 
and $\ s\ $ non-root in $ \ (A, \to), \ $ then $\ S(r)=S(s) \ $ and 
$\ s \to r .\ $ Thus $\ r \rightsquigarrow s \ $ is the first step in the reverse of an exit path.  A link $\ r \rightsquigarrow s \ $ with $\ r\ $ non-root 
and $\ s\ $ root in $ \ (A, \to), \ $ then $\ S(r)<S(s) \ $ and 
$\ \alpha^{-1}r \subseteq s ,\ $ so $\ r \ $ is a leaf. Thus $\ r \rightsquigarrow s \ $ is the last step in the reverse of an entrance path. 
A link $\ r \rightsquigarrow s \ $ with $\ r \ $ and $\ s\ $ non-roots 
in $ \ (A, \to), \ $ then $\ S(r)=S(s) \ $ and 
$\ s \to r ,\ $  Thus $\ r \rightsquigarrow s \ $ is a link either in an inner constant entropy path, or an entrance path or an exit path.  
\end{proof}

\

A \textbf{reversible} micro-macro dynamical system  $\ (X,A,m,\alpha, r) \ $ comes with a reversion map $\ r:X \to X\ $ such that $\ r^2=1\ $ and $\ \alpha^{-1}= r\alpha r.\ $ We say that $\ r \ $ is \textbf{invariant} if $\ mr=m, \ $ i.e. $\ r \ $ descends to the identity map on $\ A.\ $ We say that $\ r \ $ is \textbf{equivariant} if it descends to a bijection $\ r:A \to A \ $ such that $\ mr=rm. \ $ We say that $ \ r\ $  \textbf{preserves entropy} \ if $\ S(\alpha i)=S(i)\ $ for  $\ i\in X.\ $   Note that the various constructions on  micro-macro systems from Section \ref{mms} can be easily adapted to include reversion maps.\\

\begin{prop}\label{mi}
{\em Let $\ (X,A,m,\alpha, r) \ $ be a reversible system. The reversion
map $\ r:X \to X \ $ yields an isomorphism from $\ (X,A,m,\alpha) \ $
to $\ (X,A,m,\alpha^{-1}) \ $ if and only if $\ r \ $ is equivariant. 
}
\end{prop}

\begin{proof}[\textbf{Proof}]
Map $\ r\ $ is a bijection since it is its own inverse. For it to be an isomorphism it should satisfy $\ r\alpha = r\alpha^{-1}, \ $ which it does,
and there should be map $\ r:A \to A \ $ such that $\ rm=rm, \ $ i.e.
$\ r \ $ should be equivariant. 
\end{proof}

From Proposition \ref{mi} we learn that a micro-macro reversible system
need not be isomorphic to its inverse system, and that this happens exclusively for non-equivariant reversion maps. In Section \ref{ioa} we consider systems with entropy preserving reversion maps, and in Section \ref{ifa} we consider systems with non entropy preserving reversion maps.

\begin{prop}\label{p10}
{\em Let $\ (X,A,m,\alpha, r) \ $ be a reversible system with $\ r \ $
equivariant. Transitions $\ a \to b \ $ and $\ rb \to ra \ $
are reproducible if and only if $ \ \alpha a=b.\ $  
Under this condition, if $\ r \ $ is invariant, then we have a cycle of reproducible transitions 
$\ a \to b.$  }
\end{prop}

\begin{proof}[\textbf{Proof}]
If $\ a \to b \ $ and $\ rb \to ra \ $ are reproducible, then 
$$\ |a|\ = \ |\alpha a|\ \leq\ |b|\ = \ |rb| \ = |\alpha rb| 
\ \leq\ |ra| \ = \ |a|,\ $$ thus $\ b=\alpha a. \ $ Conversely, if $\ b=\alpha a, \ $ then
$\ \alpha rb\ = \ r\alpha^{-1}b\ = \ ra,\ $   thus  $\ a \to b \ $ and $\ rb \to ra. \ $ Assume now that $\ r \ $ is invariant, we have $\ ra=a, \ $ $\ rb=b,\ $  and transitions $\ a\to b\ $ and $\ b\to a,\ $ so we have cycle
$\ a\to b.$
  
\end{proof}

One may wish to consider a slightly weaken notion of reproducibility where for a given $\ 0 <\varepsilon < \frac{1}{2}, \ $ we say that a transition $\ a\to b  \ $ between macrostates is $\ \varepsilon$-reproducible if $\ [\alpha]_{a,b} \geq 1-\varepsilon, \ $ i.e. most microstates in $\ a \ $ move to $\ b \ $ in a unit of time, but an $\ \varepsilon\ $ minority of microstates in $\ a \ $ are allowed to move to other macrostates. Just as before one can construct a graph  $\  (A,\xrightarrow{\varepsilon})\ $ of $\ \varepsilon$-reproducible transitions. Here we limit ourselves to consider systems $\ (X,A,m, \alpha)\ $ such that $\  (A,\xrightarrow{\varepsilon})\ $ has no sinks, i.e. we assume that there is map $\ t:A \to A\ $ such that
$\ [\alpha]_{a,t(a)} \geq  1- \varepsilon \ $
for $\ a \in A.$ \\

\begin{thm}{\em Let system $\ (X,A,m, \alpha)\ $ be such that there is map $\ t:A \to A\ $ with $\ [\alpha]_{a,t(a)} \geq  1- \varepsilon. \ $ 
If $\ 0 <\varepsilon < \frac{1}{2} \ $ is such  that $\ S(a)<S(b) \ $ implies  $\ S(a)<S(b)+ \mathrm{ln}(1-\varepsilon) ,\ $ then the 
decreasing entropy ratio is $\ \varepsilon\ $ bounded,
i.e. $\ \displaystyle\frac{|D|}{|X|}\leq \varepsilon. \ $

}
\end{thm}

\begin{proof}
We have that
$\ \ \displaystyle \frac{|t a|}{|a|}\ \geq  \ 
\frac{ |\alpha a \cap ta|}{|a|} \  = \ 
 \frac{ |\llbracket a, ta\rrbracket|}{|a|} \  =
 \  [\alpha]_{a,ta} \ \geq  \ 1 - \varepsilon. \ \ $ Thus 
\begin{equation}\label{yyyy}
S(ta) \ = \  \mathrm{ln}|ta|  \ \geq \ \mathrm{ln}|a|\  + \ \mathrm{ln}(1-\epsilon)  \ =  \ S(a) \ + \ \mathrm{ln}(1-\varepsilon).
\end{equation}
If $\ S(ta) < S(a), \ $ then we would also have that 
$\ S(ta) <  S(a)+ \mathrm{ln}(1-\varepsilon) \ $ by our choice of
 $\ \varepsilon. \ $
We conclude by (\ref{yyyy}) that $\ S(ta)  \geq  S(a),\ $ and therefore
$$ \frac{|D|}{|X|}\ = \
 \sum_{\substack{a,b \in A \\ S(b)<S(a)}}
 \frac{|a|}{|X|}\frac{ |\llbracket a,b\rrbracket|}{|a|} \ = \ 
\sum_{\substack{a,b \in A \\ S(b)<S(a)}}p_a[\alpha]_{ab} \ = $$
$$\sum_{a\in A}p_a\bigg(\sum_{\substack{b \in A \\ S(b)<S(a)}}[\alpha]_{ab}\bigg)
\ \leq \  \varepsilon \sum_{a\in A} p_a \ = \ \varepsilon .$$

\end{proof}

Despite its inherent interest, reproducibility — even in its more inclusive $\varepsilon$-version — may prove too restrictive. There seems to be no a priori reason to exclude systems with, say, equally likely transitions $\ a\to b\ $ and $\ a\to c;\ $ while this sort of situations may not arise in examples of interest, one should prove this rather than assume it from the outset. In the next section we develop a formalism that allows this sort of situations.

\section{Irreversibility From Coarse-Graining}\label{icg8}

The structural maps in a micro-macro dynamical system induce convex maps between probabilities spaces as in the diagram 
\[
  \begin{tikzcd}
    \mathrm{cProb}(X) \arrow{r}{c\alpha_*} \arrow[hookrightarrow]{d}  & \mathrm{cProb}(X)  \\
     \mathrm{Prob}(X) \arrow{r}{\alpha_*} & \mathrm{Prob}(X) \arrow{u}[swap]{c} \arrow{d}{m_*}\\
     \mathrm{Prob}(A) \arrow{u}{c} \arrow{r}{[\alpha]} & \mathrm{Prob}(A)
  \end{tikzcd}
\]

The pushforward maps $\ \alpha_* \ $ and $\ m_* \ $ 
are given, respectively, by $\ (\alpha_*p)_i = p_{\alpha^{-1}i} \ $
and $\ (m_*p)_a = p(a). \ $
The coarse-graining map $\ c \ $ on  $\ \mathrm{Prob}(A)\ $ was introduced in (\ref{cgrain}), the coarse-graining map 
$\ c: \mathrm{Prob}(X)\to \mathrm{Prob}(X)\ $ is given by 
$\ \displaystyle c(q)_i= \frac{q(a)}{|a|}\ $ for $\ i\in a. \ $
Coarse-graining $\ c\ $ is an idempotent map, its image 
$\ \mathrm{cProb}(X) \ $ consists of
probability measures on $\ X\ $  homogeneous on macrostates.
Note that $\ \alpha_* \ $ preserves Shannon entropy because   
$\ \alpha \ $ is a bijection.\\

\begin{prop}\label{l5}{\em Stochastic map $\ [\alpha]:A \to A \ $ may be identified with the convex map \\
$\ [\alpha]:\mathrm{Prob}(A) \to \mathrm{Prob}(A)\ $ given by
$\ ([\alpha]q)_b= \sum_{a\in A}q_a[\alpha]_{ab}. \ $ Each  $\ q \in \mathrm{Prob}(A) \ $ defines a  Markov chain $\ (q,[\alpha]) \ $ on $\ A\ $  given by the probability law\footnote[2]{Defined on the product or cylinder $\sigma$-algebra on $\ A^{\mathbb{N}}. \ $ Maps
$\ X_n: A^{\mathbb{N}} \to A\ $ are given by $\ X_n(w)=w_n, \ $
and $\ \  \mathrm{Pr}[X_0=a_0 , \dots ,X_n= a_n]\ = \ 
\mathrm{Pr}[w\in A^{\mathbb{N}} \ | \ X_0(w)=a_0 , \dots ,X_n(w)= a_n].$ } 
on $\ A^{\mathbb{N}}:\ $ 
\begin{equation}\label{mc}
\mathrm{Pr}[X_0=a_0 , \dots ,X_n= a_n] \ = \ q_{a_0} \frac{|\llbracket a_0a_1\rrbracket| \ |\llbracket a_1a_2\rrbracket |\cdots
|\llbracket a_{n-1}a_n\rrbracket|}{|a_{0}||a_{1}|\cdots |a_{n-1}|}. 
\end{equation} 
Markov chain $\ (p,[\alpha]) \ $ is stationary. \ Furthermore
$\ \ [\alpha]=m_*\alpha_*c.\ $ 
}
\end{prop}

\begin{proof}[\textbf{Proof}]
 The first part follows from (\ref{T}).  The chain $\ (p,[\alpha]) \ $ is stationary since
\begin{equation*}
([\alpha]p)_b \ = \ \sum_{a\in A}p_a[\alpha]_{ab} \ = \ 
\sum_{a\in A}\frac{ |a||\alpha(a) \cap b |}{|X||a|} \ = \ \frac{|\alpha(\bigsqcup_{a\in A}a)  \cap b |}{|X|}
\ = \  \frac{|X \cap  b |}{|X|} \ = \ \frac{|b |}{|X|}\ = \ p_b.
\label{eq:}
\end{equation*} 
For $\ q \in \mathrm{Prob}(A) \ $ the identity $\ [\alpha]q=m_*\alpha_*cq \ $ holds since
$$(m_*\alpha_*cq)_b \ = \ (\alpha_*cq)(b) \ = \ 
(cq)(\alpha^{-1}b)\ = \ \sum_{i\in \alpha^{-1}b}(cq)_i \ = $$
$$\sum_{a\in A}\sum_{i\in a\cap\alpha^{-1}b}\frac{q_a}{|a|} \ = \ 
\sum_{a\in A}q_a\frac{|a\cap\alpha^{-1}b|}{|a|} \ = \  \sum_{a\in A}q_a[\alpha]_{ab} \  = \ ([\alpha]q)_b.$$
\end{proof}

Map $\ c\alpha_*:\mathrm{cProb}(X) \longrightarrow \mathrm{cProb}(X)\ $ is given for  by 
\begin{equation}\label{R}
(c\alpha_*q)_j \ = \ 
\frac{ \alpha_*q(b) }{|b|} \ = \ 
\frac{ q(\alpha^{-1}b) }{|b|} \ = \ 
\sum_{i\in \alpha^{-1}b}\frac{ q_i }{|b|}\ = \ 
\sum_{i\in X}q_i[c\alpha_*]_{ij},
\end{equation}
where the stochastic matrix $\ [c\alpha_*] \ $ is given by $\ \displaystyle [c\alpha_*]_{ij}=\frac{1}{|b|}\ $ if $\ \alpha(i), j \in b, \ $
and $\ [c\alpha_*]_{ij}=0 \ $ otherwise.  The uniform probability measure $\ u\ $ on $\ X\ $ is stationary under $\ [c\alpha_*]_{ij}.$

\begin{thm}\label{c13}
{\em The maps $\ c\ $ and $\ m_*\ $ yield an isomorphism between the 
stochastic maps $\ c\alpha_* \ $ and $\ [\alpha] \ $ represented by the commutative square
\[
  \begin{tikzcd}
    \mathrm{cProb}(X) \arrow{r}{c\alpha_*} \arrow{d}{m_*}  & \mathrm{cProb}(X)\arrow{d}{m_*}  \\
     \mathrm{Prob}(A) \arrow{u}{c} \arrow{r}{[\alpha]} & \mathrm{Prob}(A)\arrow{u}{c}
  \end{tikzcd}
\]
Markov chains $\ (u,c\alpha_*) \ $ and $\ (p,[\alpha]) \ $ are isomorphic. 
}
\end{thm}

\begin{proof}[\textbf{Proof}]
Clearly $\ c\ $ and $\ m_*\ $  are inverse of each other. 
The result follows, using Lemma \ref{l5}, from the identities
\begin{equation}\label{uuu}
[\alpha]m_*\ = \ m_*\alpha_*cm_* \ = \ m_*\alpha_* \ = \ m_*c\alpha_*.
\end{equation}
Since $\ [\alpha] \ $ and $\ c\alpha_* \ $ are isomorphic
and $\ p= m_*u, \ $  then $\ (p,[\alpha])\ $  
and $\ (u,c\alpha_*)\ $ are also isomorphic.
\end{proof}

\begin{thm}\label{esl}
{\em Map $\ c\alpha_* \ $ increases Shannon entropy,  i.e. 
\begin{equation}\label{Rsube}
H(q)\ \leq \ H(c\alpha_*q) \ \ \ \ \mbox{for} \ \ \ q\in \mathrm{cProb}(X).  
\end{equation}
Map $\ [\alpha] \ $ increases Shannon entropy plus mean Boltzmann entropy, i.e.  
\begin{equation}\label{Tsube}
H(q)+S(q) \ \leq \ H([\alpha]q)+S([\alpha]q) \ \ \ \ \mbox{for} \ \ \ q\in \mathrm{Prob}(A).  
\end{equation}
}
\end{thm}

\begin{proof}[\textbf{Proof 1}]
It is known that coarse graining increases Shannon entropy, indeed  for $\ q \in \mathrm{Prob}(X) \ $ identity (\ref{cq}) implies  that
$$H(q) \ = \  H(m_*q) \ + \ \sum_{a\in A}q(a)H(q|_a)  \ \leq \
H(m_*q) \ + \ \sum_{a\in A}q(a)\mathrm{ln}|a| 
\ = \ H(cm_*q) \ = \ H(cq) ,$$
where $\ q|_a \ $ is the normalized restriction of $\ q \ $ to $\ m^{-1}a. \ $ Thus we have that
\begin{equation}\label{csube}
H(q)  \ = \ H(\alpha_*q)  \ \leq \  H(c\alpha_*q).
\end{equation}
Let $\ q\in \mathrm{Prob}(A), \ $ by Theorem \ref{c13} and (\ref{cq}) we have that 
$$H(q)+S(q) \ = \ H(cq)\ \leq \ H(c\alpha_*cq) \ = \ H(c[\alpha]q) \ = \ H([\alpha]q)+S([\alpha]q) .$$
\end{proof}

\begin{proof}[\textbf{Proof 2}]  
For any finite set $\ C, \ $ the data processing inequality \cite{ct} tell us that
the Kullback-Leibler divergency $\ D(\ \| \ ):\mathrm{Prob}(C) \times \mathrm{Prob}(C) \rightarrow [0,\infty] \ $ is contractive under stochastic maps. 
We obtain (\ref{Rsube})  as follows  
$$\mathrm{ln}|X| - H(c\alpha_*q) \ = \ D(c\alpha_*q\|u) \ = \ D(c\alpha_*q\|c\alpha_*u)\ \leq \ D(q\|u) 
\ = \  \mathrm{ln}|X| - H(q).$$
We obtain (\ref{Tsube}) as follows 
$$\mathrm{ln}|X| - H([\alpha]q)-S([\alpha]q) \ = \ D([\alpha]q\|p) \ = \ $$ $$D([\alpha]q\|[\alpha]p)\ \leq \ D(q\|p) \ = \  \mathrm{ln}|X| - H(q)-S(q).$$
\end{proof}

The partitions $\ \mathrm{Par}(X) \ $ of a finite set $ \ X \ $ form a bounded lattice with the partial order $\ \sigma \leq \tau \ $ if each block of $\ \sigma \ $ is included in some block of $\ \tau. \ $ Boltzmann entropy
$\ S: \mathrm{Par}(X) \to \mathbb{R}_{\geq 0}  \ $ is increasing in this order, indeed
\begin{equation}\label{incsp}
S(\sigma) \ = \ \sum_{a\in \sigma}\frac{|a|}{|X|}\mathrm{ln}|a| 
\ = \ \sum_{b\in \tau}\bigg(\sum_{a\subseteq b,\ a\in \tau}\frac{|a|}{|X|}\mathrm{ln}|a|\bigg) \ \leq 
\end{equation}
$$\sum_{b\in \tau}\bigg(\sum_{a\subseteq b,\ a\in \tau}\frac{|a|}{|X|}\bigg)\mathrm{ln}|b|
\ = \ \sum_{b\in \tau}\frac{|b|}{|X|}\mathrm{ln}|b| \ = \ S(\tau). $$
The joint $\ \sigma \vee \tau \ $  is the smallest partition bigger than
$\ \sigma \ $ and $\ \tau. \ $ \\

\begin{thm}\label{mct}
{\em Consider system $ \ (X,A,m,\alpha)\  $  and its associated Markov chain with transition map $\ [\alpha]:A \to A. \ $   Let $\ \Gamma \ $ be the graph with vertex set $\ X \ $ such that  there is a $\ \Gamma$-link $\ i \to j \ $ if and only if $\ \alpha i = j \ $ or  $\ i,j\in a\ $ for some $\ a\in A.  $ \\

\noindent a) Macrostate $\ b\in A \ $ is $\ [\alpha]$-reachable from macrostate $\ a\in A \ $
if and only if there is a path in $\ \Gamma \ $ beginning at $\ a \ $ and
ending at $\ b. \ $  A macrostate is $\ [\alpha]$-absorbing if and only if it is a union of $\ \alpha$-cycles. \\

\noindent b) There are no $\ [\alpha]$-transient macrostates, i.e. if there is a
$\ \Gamma$-path from $\ a\ $ to $\ b,\ $ then there is also
a $\ \Gamma$-path from $\ b\ $ to $\ a.\ $ \\

\noindent c) The partition in $\ [\alpha]$-communication classes equals the joint of the partition into $\ \alpha$-cycles, and the partition determined $\ m, \ $ i.e. $\ \mathrm{Comm}_{[\alpha]} \ = \ \mathrm{Cyc}_{\alpha} \vee \pi.\ $
Furthermore we have that 
\begin{equation}\label{dis1}
(X,A,m,\alpha)\ = \  
\coprod_{C \in \mathrm{Comm}_{[\alpha]} }\big(\widehat{C},C,m,\alpha \big) \ \ \ \ \ \mbox{where} \ \ \ \ \widehat{C}=\coprod_{a\in C}a.
\end{equation} 
The restriction $\ [\alpha]_C \ $ of $\ [\alpha] \ $ to each $ \ C \in \mathrm{Comm}_{[\alpha]} \ $ is irreducible.\\

\noindent d) $ [\alpha]_C \ $ is aperiodic if the greatest common divisor of $\ \Gamma$-cycles  lengths is $\ 1. \ $  In this case,  for any $\ q \in \mathrm{Prob}(A)\ $
and $\ a \in C\ $ we have that
\begin{equation}\label{cac}
\big[\lim_{n \to \infty} [\alpha]^nq_C\big]_a \ = \  q(C)\frac{|a|}{|\widehat{C}|},
\end{equation} 
where measure $\ q_C\in \mathrm{Meas}(A)\ $ is equal to $\ q\ $ on $\ C, \ $ and to $\ 0 \ $ on $\ A\setminus C.$ \\

\noindent e) $ [\alpha]_C \ $ has period $\ d \ $ if the greatest common divisor of 
$\ \Gamma$-cycles lengths is $\ d. \ $
In this case, $\ C\ $ decomposes as $\ C=\coprod_{i=1}^d D_i ,\ $
with macrostates moving cyclically from $\ D_i\ $ to $\ D_{i+1}\ $ under $\ [\alpha]_C. \ $ For $\ q \in \mathrm{Prob}(A) \ $ and $\ a \in D_i \ $  
we have 
\begin{equation}\label{cacpe}
\big[\lim_{n \to \infty} [\alpha]^{i+nd}q_{D_i}\big]_a \ = \  q(D_i)\frac{|a|}{|\widehat{D_i}|}
 \end{equation} 
 where measure $\ q_{D_i}\in \mathrm{Meas}(A)\ $ is equal to $\ q\ $ on $\ D_i, \ $ and to $\ 0 \ $ on $\ A\setminus D_i.$}
\end{thm}

\begin{proof}[\textbf{Proof}]
a) $ b \in A\ $ is reachable from $\ a \in A \ $ if there is $\ n\geq 0\ $ such
that $\ [\alpha]^{n}_{ab}>0.  \ $ If $\ n=0, \ $ we must have $\ a=b, \ $ and we have
link $\ i\to i\ $ in $\ \Gamma \ $ for $\ i\in a. \ $
If $\ n>0, \ $ from (\ref{mc}) we get that there are microstates $\ i_s \in a_s \ $ such that $\ \alpha(i_s)\in a_{s+1}\ $ where $\ 0\leq s \leq n-1,\ $ $\ a_s\in A, \ $ $\ a_0=a, \ $ and $\ a_n=b.\ $ Thus we have path 
$\ i_0\to \alpha(i_0) \to i_1\to \alpha(i_1)\to 
\cdots \to i_{n\mathrm{-}1}\to \alpha (i_{n\mathrm{-}1})\ $ in $\ \Gamma \ $
with $\ i_0 \in a\ $ and $\ \alpha (i_{n\mathrm{-}1}) \in b. \ $ 

The converse follows by induction. Suppose we have a $\ \Gamma$-path from $\ a \ $ to $\ b \ $ of length $\ n+1.\ $ If the last step is a link
$\ i\to j \ $ between two macrostates in $\ b,\ $  then by induction
 $\ [\alpha]_{a,b}^p >0 \ $ for some $\ p\geq 0. \ $ If the last link
is of the form $\ i \to \alpha i \ $ with $\ i \in c, \ $ then by induction
there is $\ [\alpha]_{a,c}^p >0 \ $ for some $\ p\geq 0, \ $ and
$\ [\alpha]_{c,b}>0,\ $ thus $\ [\alpha]_{a,b}^{p+1} >0. \ $ 

Macrostate $\ a \in A \ $ is absorbing if $\ [\alpha]_{ab}=0 \ $ for $\ b\neq a, \ $
i.e. $\ \alpha a \subseteq a. \ $ So the $\ \alpha$-cycle
of $\ i\in a \ $ is included in $\ a,\ $  thus $\ a \ $ is a union of   $\ \alpha$-cycles.\\

\noindent b) We use induction on the length $\ n\ $ of 
$\ \Gamma$-paths from $\ a \ $ to $\ b. \ $  For $\ n=0 \ $ we again get $\ a=b. \ $
Assume we have $\ \Gamma$-path 
$\ i_0\to i_1\to 
\cdots \to i_{n-1} \to i_{n}\ $  with
$\ n>0. \ $ By induction we have $\ \Gamma$-path in   from
$\ i_{n-1}\ $ to $\ i_{0},\ $ so it remains to find a path from
$\ i_{n}\ $ to $\ i_{n-1}.\ $ If $\ i_{n},\ i_{n-1}\ $ belong to the same macrostate, then there is a $\Gamma$-link from $\ i_{n}\ $ to $\ i_{n-1},\ $ otherwise
$\ \alpha (i_{n-1})=i_{n}\ $ and there is $\ k\geq 1 \ $ such that
$\ \alpha^k (i_{n})=i_{n-1},\ $ and thus we have $\Gamma$-path 
$\ i_n\to \alpha(i_n) \to 
\cdots \to \alpha^k (i_{n})=i_{n-1}.\ $  \\

\noindent c) Macrostates $\ a,\ b \ $ belong the same block in $\ \mathrm{Comm}_{[\alpha]}  \ $ if and only if there is a $\Gamma$-path from $\ i \in a\ $ 
to $\ j\in b.\ $ The result follows since each step in such a path
is of the form $\ s\to \alpha (s),\ $ linking two elements in a $\ \alpha$-cycle,
or of the form $\ s \to t\ $  with $\ s,\ t \ $ in some block of $\ \pi.$ \\

\noindent d) and e) follow from the fundamental theorems for finite Markov chains \cite{no}, and the fact that distribution $\ \displaystyle p_a=\frac{|a|}{|X|}\ $ on $\ A \ $ is stationary under $\ [\alpha].$
\end{proof}

\begin{thm}\label{revv}
{\em Let $ \ (X,A,f,\alpha, r)\  $  be a reversible system.\\

\noindent a) For non-empty subsets $\ a,b \subseteq X \ $ we have that
\begin{equation}\label{trev}
|\llbracket a,b\rrbracket |\ = \ |\llbracket rb,ra\rrbracket | \ \ \ \mbox{and \ thus} \ \ \  [\alpha]_{ab} \ = \ e^{S(b)-S(a)}[\alpha]_{rb,ra}.
\end{equation}
b)  $ (p,[\alpha^{-1}]) \ $ is the reversed Markov chain of $\ (p,[\alpha]). \ $ \\

\noindent c) Markov chain $\ (p,[\alpha]) \ $ is reversible if and only if
$\ |\llbracket a,b \rrbracket |=|\llbracket b,a \rrbracket |  \ $ for $\ a,b \in A.$ \\

\noindent d) If $\ r \ $ is invariant, then Markov chain $\ (p,[\alpha])\ $ is a reversible.   
}
\end{thm}

\

\begin{proof}[\textbf{Proof}] 
a) The idempotent map $\ r\alpha \ $ induces bijective maps 
$$ r\alpha : \llbracket a,b \rrbracket \leftrightarrows \llbracket rb, ra \rrbracket : r\alpha.$$
So we have that
$$[\alpha]_{a,b} \ = \ \frac{|\llbracket a,b \rrbracket}{|a|}\ = \ \frac{|b|}{|a|}\frac{|\llbracket a,b \rrbracket|}{|b|}
\ = \ \frac{|b|}{|a|}\frac{|\llbracket rb,ra \rrbracket  |}{|rb|}
\ = \  \frac{|b|}{|a|}[\alpha]_{rb,ra} \ = \ e^{S(b)-S(a)}[\alpha]_{rb,ra}.$$
b) By definition the reversed Markov chain $\ (p,[\alpha]^r) \ $ of $\ (p,[\alpha])\ $ is given by $$[\alpha]^r_{a,b}\ = \ \frac{p_b}{p_a}[\alpha]_{b,a} \ = \ \frac{|b|}{|a|}[\alpha]_{b,a} \ = \ 
\frac{|b|}{|a|}\frac{|\alpha(b)\cap a|}{|b|} \ = \
\frac{|\alpha^{-1}(a)\cap b|}{|a|} \ = \ [\alpha^{-1}]_{a,b}.  $$ 
c) $(p,[\alpha]) \ $ is reversible if and only if  $\ [\alpha]^r_{a,b}=[\alpha]_{a,b}, \ $ i.e. if and only if
$$|\llbracket a,b \rrbracket| \ = \ |a|[\alpha]_{ab}=|b|[\alpha]_{b,a} \ = \ |\llbracket b,a \rrbracket|.$$  

\noindent d) Follows from item a) and c).
\end{proof}

Theorem \ref{revv}-a generalizes Proposition \ref{p10} which covers the case
$$\ [\alpha]_{a,b} \ = \ [\alpha]_{rb,ra} \ = \ 1.\ $$  The processes studied in this section may be thought of as Markov approximations to the stochastic processes directly relevant to addressing the second law. The study of these sorts of processes goes back to Kolmogorov in the context of ergodic theory \cite{aa, ct}, and will be our focus in the next section.

\section{Non-Markovian Stochastic Process}\label{nmsp}

Given system $ \ (X,A,m,\alpha)\  $ and $\ a_0, \dots , a_n \in A \ $ set  
\begin{equation}
\llbracket a_0, \dots, a_n \rrbracket \ = \ \{i\in a_0 \ | \ \alpha^t (i) \in  a_t \ \ \ \mbox{for}
\ \ \ 0\leq t \leq n \} \ \subseteq \ a_0.
\label{eq:path}
\end{equation}
This notation makes sense for arbitrary subsets of $\ X. \ $ The probability that a microstate in $\ a_0\ $ moves consequently to macrostates
$\ a_1, \dots, a_n \ $ is given by
\begin{equation}
[\alpha]_{a_0, \dots, a_n}\ = \ \frac{|\llbracket a_0, \dots, a_n
\rrbracket | }{|a_0|}
\label{eq:path}
\end{equation} 

\begin{thm}\label{psp}{\em  Consider system $ \ (X,A,m,\alpha)\ $ with return time $\ d. \ $  For each $\ q\in \mathrm{Prob}(A) \ $ we have stochastic process  $\ (q,[\alpha]_{\bullet}) \ $  on $\ A\ $  with probability law
on $\ A^{\mathbb{N}}\ $ given by 
\begin{equation}\label{mc2}
\mathrm{Pr}[X_0=a_0 , \dots ,X_n= a_n] \ = \ 
q_{a_0} [\alpha]_{a_0,\dots,a_n}\ = \ 
q_{a_0}  \frac{|\llbracket a_0,\dots,a_n \rrbracket |}{|a_{0}|}. 
\end{equation} 
a) Kolmogorov consistency axioms hold. \\

\noindent b) $\mathrm{Pr}[X_0= b] =  q_b\ \  $ and $  \ \ \mathrm{Pr}[X_n= b] = \big([\alpha^{n}]q\big)_b\ \ $ for $\ \ n\geq 1. \ $
Therefore we have that $\ \mathrm{Pr}[X_{i+nd}=b ]\ = \ \mathrm{Pr}[X_{i}=b ]. \ $ \\

\noindent c) $\displaystyle \mathrm{Pr}[X_{n+1}=a_{n+1} \ | \ X_0=a_0 , \dots ,X_n= a_n] \ = \ 
 \frac{|\llbracket a_0,\dots,a_n,a_{n+1} \rrbracket|}{|\llbracket a_0,\dots,a_n \rrbracket |}.$\\
 
\noindent d) Process $\ (q,[\alpha]_{\bullet}) \ $  has period $\ d. \ $
For $\ d=1\ $ it is a Markov process.
For $\ d \geq 2\ $ it is a Markov process if and only
if $\ \alpha \ $ is equivariant, i.e. $\ \alpha \ $ descends to a bijection
$\ \alpha:A \to A \ $ such that $\ m\alpha = \alpha m.$ \\

\noindent e) Process $\ (p,[\alpha]_{\bullet}) \ $ is strictly stationary. \\ 

}
\end{thm}

\begin{proof}[\textbf{Proof}]
a) Follows from the identities
$$\sum_{a_{k+1},\dots,a_{n} \in A}\mathrm{Pr}[X_0=a_0 , \dots ,X_k=a_k,X_{k+1}
\ = \ a_{k+1} \dots X_n= a_n]\ \ = $$
$$ \sum_{a_0 \in A}q_{a_0}  \frac{|[a_0,\dots,a_k,\coprod_{a_{k+1} \in A} a_{k+1},\dots,\coprod_{a_{n} \in A} a_{n}]|}{|a_{0}|}\ \ = \ \
\sum_{a_0 \in A}q_{a_0}  \frac{|\llbracket a_0,\dots,a_k,X,\dots,X
\rrbracket |}{|a_{0}|} \ \ =$$ 
$$\sum_{a_0 \in A}q_{a_0}  \frac{|\llbracket a_0,\dots,a_k \rrbracket |}{|a_{0}|}
\ \ = \ \ \mathrm{Pr}[X_0=a_0 , \dots ,X_k=a_k]. $$
\noindent b) Marginal distributions of $\ (q,T_{\bullet}) \ $ are given
for $\ n\geq 1 \ $ by
$$\mathrm{Pr}[X_n= b]  \ = \ 
\sum_{a_0,\dots,a_{n-1} \in A}q_{a_0}  \frac{|\llbracket a_0,a_1, \cdots,a_{n-1},a_n \rrbracket|}{|a_{0}|}\ \ =  $$
$$ \sum_{a_0 \in A}q_{a_0}  \frac{|\llbracket a_0,X, \dots, X, b
\rrbracket |}{|a_{0}|}\ \ = \ \
\sum_{a_0 \in A}q_{a_0}[\alpha^{n}]_{a_0,b}\ \ = \ \ \big([\alpha^{n}]q\big)_b . \ $$

\noindent c) Follows from directly from identity (\ref{mc2}). \\

\noindent d) We show $\ d$-periodicity. Let $\ n \geq d, \ \ $ 
if $\ \ a_j=a_i\ \ $  for $\ \ j=i \ \  \mathrm{mod} \ \ d, \ $  then 
$$\mathrm{Pr}[X_0=a_0 , \dots ,X_n= a_n] \ = \
\mathrm{Pr}[X_0=a_0 , \dots ,X_{d-1}= a_{d-1}], $$  
otherwise $\ \mathrm{Pr}[X_0=a_0 , \dots ,X_n= a_n]\ = \ 0.\ $ Indeed 
$\ |\llbracket a_0 , \dots ,a_n \rrbracket | \ = \ 
|\llbracket a_0 , \dots ,a_{d-1} \rrbracket |\ $ if $\ a_j=a_i\ $ for $\ j=i\ $ mod $\ d, \ $ and 
otherwise $\ \llbracket a_0, \dots ,a_n \rrbracket =\emptyset.\ $ \\

\noindent Let $\ d=1, \ $ then $\ \alpha \ $ is the identity map, and $\ (q,[\alpha]_{\bullet}) \ $ is equal to the Markov process $\ (q,\delta)\ $ where the transition matrix
$\ \delta \ $ is the Kronecker delta.  Let $\ d=2 \ $ 
and assume that $\ (q,T_{\bullet}) \ $  is Markov. The Chapman-Kolmogorov equation give
$$\delta_{a,c}\ = \ \mathrm{Pr}[X_{d}=c \ | \ X_{0}= a] \ = \ 
\sum_{b\in A}\ \mathrm{Pr}[X_{d}=c \ | \ X_{1}= b] \ \mathrm{Pr}[X_{1}=b \ | \ X_{0}= a].$$
Matrix $\ [\alpha]_{a,b}=\mathrm{Pr}[X_{1}=b \ | \ X_{0}= a]\ $ and its inverse
$\ [\alpha]^{-1}_{a,b}=\mathrm{Pr}[X_{d}=b \ | \ X_{1}= a] \ $ are both stochastic, so $\ [\alpha]_{a,b} \ $ is necessarily a permutation matrix, i.e.
$\ \alpha\ $ descends to a permutation on $\ A.\ $ Conversely, if $ \ \alpha \ $
descends to a permutation on $\ A,\ $ then process $\ (q,[\alpha]_{\bullet}) \ $ is equal to the Markov process $\ (q,\delta_{\alpha(a),b}). \ $   \\

\noindent e) Consider the stochastic process 
$\ (p,[\alpha]_{\bullet}), \ $ for $\ k \geq 1\ $ we have that 
$$\mathrm{Pr}[X_{k}=a_0 , \dots ,X_{n+k}= a_n] \ = \ 
\sum_{b_0,\dots,b_{k-1}\in A}
\frac{|\llbracket b_0,\dots,b_{k-1},a_0,\dots,a_n \rrbracket |}{|X|}\ =
$$ 
$$\frac{|\llbracket X,\dots,X,a_0,\dots,a_n \rrbracket |}{|X|} \ = \ 
\frac{|\llbracket a_0,\dots,a_n \rrbracket |}{|X|} \ = \ \mathrm{Pr}[X_{0}=a_0 , \dots ,X_{n}= a_n].$$
\end{proof}

\begin{rem}
{\em Consider the stochastic process on $\ A^{\mathbb{N}}\ $ with
variables $\ X_n \ $ independent and identically distributed as
$\ \mathrm{Pr}[X_n=a] = p_a,\ $ thus  
 \begin{equation}\label{mc20}
\mathrm{Pr}[X_0=a_0 , \dots ,X_n= a_n] \ = \ 
p_{a_0} \cdots p_{a_n}. 
\end{equation} 
This is a Markov process with transition matrix $\ T_{ab} = p_b. \ $
The rows of $\ T\ $ are equal, therefore $\ T \ $ is a permutation matrix if and only if $\ |A|=1.\ $ So by Theorem \ref{psp}-d, if $\ |A|\neq 1,\ $ there is no system $\ (X,A,m,\alpha)\ $ such that $\ (p, [\alpha]_{\bullet})\ $ reproduces the stochastic process (\ref{mc20}). 
}
\end{rem}

\begin{thm}\label{revv}
{\em Let $ \ (X,A,f,\alpha, r)\  $  be a reversible system.\\

\noindent a) For non-empty subsets $\ a_0,\dots, a_n \subseteq X \ $ we have that
\begin{equation}\label{trev2}
|\llbracket a_0,\dots, a_n \rrbracket |\ = \ 
|\llbracket ra_n,\dots, ra_0 \rrbracket| \ \ \ \ \mbox { and thus} $$
$$[\alpha]_{a_0,\dots,a_n} \ = \ e^{S(a_n)-S(a_0)}[\alpha]_{ra_n,\cdots, ra_0}.
\end{equation}
b)  $ (p,[\alpha^{-1}]_{\bullet}) \ $ is the reversed stochastic process of $\ (p,[\alpha]_{\bullet}). \ $ \\

\noindent c) Stochastic process $\ (p,[\alpha]_{\bullet}) \ $ is reversible if and only if 
$$\ |\llbracket a_0,\dots,a_n \rrbracket |\ = \
 | \llbracket a_n,\dots,a_0 \rrbracket | \  \ \ \ \mbox{for} \ \ \ a_0,\dots,a_n \in A.$$

\noindent d) If $\ r \ $ is invariant, then the stochastic process
$\ (p,[\alpha]_{\bullet})\ $ is a reversible.   
}
\end{thm}

\begin{proof}[\textbf{Proof}] 
a) The idempotent map $\ r\alpha^n \ $ induces bijective maps 
$$ r\alpha^n : \llbracket a_0,\dots,a_n \rrbracket \leftrightarrows 
\llbracket ra_n,\dots, ra_0 \rrbracket : r\alpha^n.$$

\noindent b) Follows from the identities
$$\mathrm{Pr}_{\small{(p,[\alpha^{-1}]_{\bullet})}}[X_{0}=a_0 , \dots ,X_{n}= a_n] 
\ = \ \frac{|\{i\in a_0 \ | \ \alpha^{-t}(i) \in a_t \ \ \ \mbox{for}
\ \ \ 0\leq t \leq n \} |}{|X|}\ = $$
$$\frac{| \llbracket a_n,\dots,a_0 \rrbracket |}{|X|}\ = \
\mathrm{Pr}[X_{0}=a_n , \dots ,X_{0}= a_0] .$$
c) $(p,[\alpha]_{\bullet}) \ $ is reversible if and only if it is the same process
as $\ (p,[\alpha^{-1}]_{\bullet}),\ $ i.e. if and only if
$ \ \mathrm{Pr}[X_{0}=a_0 , \dots ,X_{0}= a_n] \ = \ \mathrm{Pr}_{\small{(p,[\alpha^{-1}]_{\bullet})}}[X_{0}=a_0 , \dots ,X_{n}= a_n],\ $ or equivalently if and only if 
$\ \displaystyle \frac{|  \llbracket a_0,\dots,a_n  \llbracket |}{|X|}\ = \ \frac{| \llbracket a_n,\dots,a_0  \rrbracket |}{|X|}.\ $ \\

\noindent d) Follows from items a) and c).
\end{proof}

\section{Irreversibility on Average}\label{ioa}

Suppose we have a random variable $\ X \in \mathbb{R}, \ $ to study its positivity  one may look at $\ \mathrm{Pr}[X>0] \ $ or $\ E(X). \ $
Though related, they reflect distinct approaches, with no direct implication between them. If one considers  
$\ X=S(\alpha^ni)-S(\alpha^{n-1}i), \ $ the entropy difference
at times $\ n\ $ and $\ n-1, \ $ then $\ \mathrm{Pr}[X>0] \ $ 
gives the probability of strictly increasing entropy, while 
$\ E(X) \ $ gives the entropy production, both computed at time $\ n.$ \\

Consider system $\ (X,A,m, \alpha). \ $   The $\ n$-\textbf{steps entropy production rate} is the map\\
 $\ \sigma_n:X \to \mathbb{R}\ $ given by
\begin{equation}\label{+}
\sigma_n(i) \ = \ \frac{1}{n}\big(S(\alpha^ni)-S(i)\big) \ = \ 
\frac{1}{n}\mathrm{ln}\bigg(\frac{|\alpha^ni|}{|i|} \bigg). \ \ \ \ \ \mbox{Thus}
\end{equation}
\begin{equation}\label{+1}
|\alpha^ni|= e^{n\sigma_n(i)}|i|  \ \ \ \
\mbox{and} \ \ \ \ \sigma_n(i) \ = \ \frac{1}{n}\sum_{k=1}^{n}\sigma_1(\alpha^{k-1}i).
\end{equation}
The $\ q$-\textbf{mean} $n$-\textbf{steps entropy production rate}   $\ \ \sigma_n^q \in  \mathbb{R}\ \ $ is given by
\begin{equation}\label{+4}
\sigma_n^q \ = \ \langle \sigma_n\rangle_{cq} \ = \ \sum_{i\in X}(cq)_i\sigma_n(i) \ \ \ \ \mbox{for} \ \ \ \ q\in \mathrm{Prob}(A). 
\end{equation}
For $\ n\geq 1,\ $ the $\ q$-\textbf{mean} \textbf{entropy production rate at time} $\ n \ $ is given by    
\begin{equation}\label{ttt1}
\Delta_n^q \ = \ \langle \sigma_1\alpha^{n-1}\rangle_{cq} \ = \ S\big([\alpha^n]q\big)  - S\big([\alpha^{n-1}]q\big).
\end{equation}
We have that 
\begin{equation}\label{+5}
\sigma_n^q \ = \ \frac{1}{n}\big(S\big([\alpha^n]q\big)  - S(q) \big)
\ = \ \frac{1}{n}\sum_{k=1}^{n}\bigg(S\big([\alpha^k]q\big)  - S\big([\alpha^{k-1}]q\big)  \bigg)
\ = \ \frac{1}{n}\sum_{k=1}^{n}\Delta_k^q. 
\end{equation}
Indeed from (\ref{+4}) we get
\begin{equation}\label{+6}
\sigma_n^q \ = \  \frac{1}{n}\sum_{a\in A, i\in a}\frac{q_a}{|a|}\big(S(\alpha^ni)-S(a)\big) \ = \
\frac{1}{n}\sum_{a\in A}q_a
\bigg(\sum_{b\in A}\frac{|\{ i \in a \ |\ \alpha^n i \in b\}|}{|a|}S(b)-S(a)\bigg)\ =
\end{equation}
\begin{equation*}\label{+7}
\frac{1}{n}\sum_{a\in A}q_a
\bigg(\sum_{b\in A}[\alpha^n]_{ab}S(b)-S(a)\bigg)\ = \
\frac{1}{n}\bigg(S\big([\alpha^n]q\big)  - S(q) \bigg).
\end{equation*}
In particular for $\ \displaystyle p_a=\frac{|a|}{|X|}\ $ we get that 
\begin{equation}\label{+90}
\Delta_n^p = S\big([\alpha^n]p\big)  - S\big([\alpha^{n-1}]p\big)
\ = \ S(p)-S(p) \ = \ 0,
\end{equation}
\begin{equation}\label{+60}
\sigma_n^p\ = \ \displaystyle
\frac{1}{n}\big(S\big([\alpha^n]p\big)  - S(p) \big) 
\ = \ \frac{1}{n}\big(S(p)- S(p)\big)   \ = \ 0. 
\end{equation}

\

\begin{thm}
{\em  Consider system $\ (X,A,m, \alpha). \ $   \\

\noindent a) $ \sigma_n^q  \geq  0 \ $ for all $ \ q\in \mathrm{Prob}(A) \ $ if and only if $\ \alpha^n \ $ preserves entropy.\\

\noindent b) $\Delta_n^q \geq  0 \ $ for all $ \ q\in \mathrm{Prob}(A) \ $ if and only if the vectors $\ \big(\ [\alpha^n]_{ab}-[\alpha^{n-1}]_{ab}\ \big)_{b\in A}\ $ in $\ \mathbb{R}^A \ $ lie in the same half-space as the vector 
$\ (\ S(b)\ )_{b\in A}\in \mathbb{R}^A \ $ with respect to the plane orthogonal
to $\ (\ S(b)\ )_{b\in A}. $ \\

\noindent c) If $\ \alpha\ $ has return time $\ d,\ $ then $\ \sigma_d=0\ $ and thus $\ \sigma_d^q  =  0 \ $ for all $ \ q\in \mathrm{Prob}(A). \ $
Either $\ \Delta_k^q =  0 \ $ for $ \ 1\leq k \leq d,\ $ or
there exists a $ \ 1\leq k \leq d\ $ such that  $\ \Delta_k^q <  0. \ $

}
\end{thm}

\begin{proof}[\textbf{Proof}]
a) If $\ \sigma_n^q  \geq  0 \ $ for all $\ q, \ $ then 
\begin{equation}\label{-1}
\sum_{a\in A}q_a
\bigg(\sum_{b\in A}[\alpha^n]_{ab}S(b)-S(a)\bigg) \ \geq \ 0 
\end{equation}
for all  $\ q\in \mathrm{Prob}(A),\ $ and therefore
\begin{equation}\label{+9}
S(a)\ \leq \ \sum_{b\in A}[\alpha^n]_{ab}S(b) \ \ \ \mbox{for \ all} \ \ a \in A.
\end{equation}
Let $\ a \ $ be a macrostate of maximum entropy. From (\ref{+9}) we conclude that 
$ \ [\alpha^n]_{ab} = 0 \ $ if $\ S(b)<S(a), \ $ i.e. $\ \alpha^n \ $ sends microstates in $\ a \ $ into macrostates with the same entropy as $\ a. \ $
Thus the zone of highest entropy is a union of $\ \alpha$-cycles. By a similar argument we see that the zone of second highest entropy is also a union of 
$\ \alpha$-cycles, an so on. Conversely, if $\ \alpha^n\ $ preserves entropy, then $\ \ \displaystyle \sigma_n(i) =   \frac{1}{n}\big(S(\alpha^ni)-S(i)\big) = 0. \ \ $ \\

\noindent b) $\ \Delta_n^q \geq  0 \ $ for all $ \ q\in \mathrm{Prob}(A) \ $
if and only if
\begin{equation}\label{+20}
\sum_{a\in A}q_a\bigg(\sum_{b\in A}\big( [\alpha^n]_{ab}-[\alpha^{n-1}]_{ab}\big)S(b)\bigg) \ \geq \ 0 \ \ \ \ \mbox{for \ all} \ \  \ q\in \mathrm{Prob}(A). 
\end{equation}
Equivalently, for each $\ a\in A\ $ we have that
\begin{equation}\label{+12}
\sum_{b\in A}\big( [\alpha^n]_{ab}-[\alpha^{n-1}]_{ab}\big)S(b) \ \geq \ 0.
\end{equation}
So the inner product of 
$\ \big(\ [\alpha^n]_{ab}-[\alpha^{n-1}]_{ab}\ \big)_{b\in A} \in \mathbb{R}^A \ $ and $\ (\ S(b) \ )_{b\in A}\in \mathbb{R}^A \ $ is non-negative;
geometrically this means that  $\ \big(\ [\alpha^n]_{ab}-[\alpha^{n-1}]_{ab}\ \big)_{b\in A} \in \mathbb{R}^A \ $ lies in the same half-space as  $\ (\ S(b) \ )_{b\in A}\in \mathbb{R}^A \ $ 
with respect to the plane orthogonal
to $\ (\ S(b)\ )_{b\in A}. \ $ \\

\noindent c) The first statement is clear $\ \ \displaystyle \sigma_d(i) \ = \   \frac{1}{n}\big(S(\alpha^di)-S(i)\big) \ = \ \frac{1}{n}\big(S(i)-S(i)\big) \ = \ 0. \ \ $ The second statement follows from the identity
$\ \displaystyle \sigma_d^q  \ = \ \frac{1}{n}\sum_{k=1}^{d}\Delta_k^q.$

\end{proof}

For $\ q\in \mathrm{Prob}(A) \ $ we let  $\ \mathrm{w}_n^q :\mathbb{R}\to [0,1]\ $ be the density distribution of the random variable $\ \sigma_n: X \to \mathbb{R} \ $ with respect to the coarse grained probability measure $\ cq \in \mathrm{Prob}(X). \ $  Results such as Theorem \ref{cr1}-b below are known as fluctuation theorems, see  Evans and Searle \cite{esl}, 
Deward and Maritan \cite{de2}.\\

\begin{thm}\label{cr1}
{\em  Consider system $\ (X,A,m, \alpha,r) \ $  such that  $\ r \ $ preserves entropy.  Let $\ q\in \mathrm{Prob}(A) \ $ be constant on the communications classes of $\ [r\alpha^n]. \ $ \\

\noindent a) In $\ n$-steps, entropy production at $\ r\alpha^ni \ $ is the opposite of entropy production at $ \ i, \ $ that is $\ \ \sigma_n(r\alpha^ni) \ = \  - \sigma_n(i)\ \ $ and thus $\ \ |r\alpha^ni|= e^{-n\sigma_n(r\alpha^n i)}|i|.$ \\  

\noindent b) $ \ \mathrm{w}_n^q(x)\ = \ e^{nx}\mathrm{w}_n^q(-x) \ \ $
for $\ \ x\in \mathbb{R}.$\\

\noindent c)  $\ \displaystyle
 \sigma_n^q \ = \sum_{x\in \mathbb{R}_{> 0}}x\big( 1 - e^{-nx} \big)\mathrm{w}_n^q(x) \ \geq \ 0. \ $ 

\noindent d) $\ \sigma_n^q  =  0 \ $ if and only if 
$\ \alpha^n \ $ preserves entropy on 
$\ \displaystyle  \coprod_{a\in A,\ q_a\neq 0}a.$ 

\noindent e) For the uniform probability $\ \displaystyle u_a=\frac{1}{|A|}\ $ on $\ A, \ $
we have that $\ \sigma_n^u  >  0 \ $ if and only if there is a microstate such that
$\ S(\alpha^n i ) \neq S(i). \ $
}
\end{thm}

\begin{proof}[\textbf{Proof}]
a) Follows from the identities
\begin{equation}\label{+3}
\sigma_n(r\alpha^ni)\ = \ \frac{S(\alpha^n r\alpha^ni)-S(r\alpha^ni)}{n} \ = \
\frac{S(ri)-S(r\alpha^ni)}{n}  \ = \ - \sigma_n(i) .
\end{equation}
Therefore using  (\ref{+1}) we get
\begin{equation}\label{+21}
|r\alpha^ni| \ = \ |\alpha^n i|\ = \  e^{n\sigma_n(i)}|i| \ = \
 e^{-n\sigma_n(r\alpha^n i)}|i|.
\end{equation}

\noindent b) Below we use to change of variables $\ i \to r\alpha^n i\ $
and item a)\ :
\begin{equation}\label{+20}
\mathrm{w}_n^q(x)\ = \ \sum_{\sigma_n(i)=x}(cq)_i \  = \
\sum_{\sigma_n(i)=x}\frac{q_{mi}}{|i|} \ = \ 
\sum_{\sigma_n(r\alpha^ni)=x}\frac{q_{m( r\alpha^n i)}}{| r\alpha^n i|}\ = 
\end{equation}
\begin{equation*}
\sum_{\sigma_n(r\alpha^ni)=x}\frac{q_{m( r\alpha^n i)} e^{n\sigma_n(r\alpha^n i)}}{| i|}\ = \ 
e^{nx}\sum_{\sigma_n(r\alpha^ni)=x}\frac{q_{m( r\alpha^n i)} }{| i|} \ = \ 
e^{nx}\sum_{\sigma_n(i)=-x}\frac{q_{mi}}{|i|} \ =\ e^{nx}\mathrm{w}_n^q(-x),  
\end{equation*}
where $\ q_{m( r\alpha^n i)} = q_{m(i)}\ $ since, by definition, 
macrostates $\ m(i)\ $ and $\ m( r\alpha^n i)\ $ belong to same communication
class of $\ [r\alpha^n]. \ $\\

\noindent c) We have finite sums
\begin{equation*}
\sigma_n^q  \ = \ \langle \sigma_n \rangle_{cq} \ = \ 
\sum_{x\in \mathbb{R}}x\mathrm{w}_n^q(x) \ = \ 
\sum_{x\in \mathbb{R}_{>0}}
x\big( \mathrm{w}_n^q(x) - \mathrm{w}_n^q(-x) \big) \ = \
\end{equation*}
\begin{equation*}
\sum_{x\in \mathbb{R}_{> 0}}x\big( \mathrm{w}_n^q(x) - e^{-nx}\mathrm{w}_n^q(x) \big) \ = \ 
\sum_{x\in \mathbb{R}_{> 0}}x\big( 1 - e^{-nx} \big)\mathrm{w}_n^q(x).
\end{equation*}

\noindent d) From c) we have that $ \ \sigma_n^q  = 0 \ $ if and only if 
 $\ \mathrm{w}_n^q(0)=1\ $ and 
$$ 0 \ = \  \mathrm{w}_n^q(x) \ = \ \sum_{\sigma_n(i)=x}\frac{q_{mi}}{|i|}
\ \ \ \ \mbox{for} \ \ \ \ x\neq 0. $$
Thus $\ q_{a}=0\ $ if there is microstate $\ i\in a\ $ such that 
$\ S(\alpha^n i)\neq S(i).\ $ \\
 
\noindent e) Uniform probability $\ \displaystyle u_a= \frac{1}{|A|}\ $ satisfies the conditions of the theorem, thus the previous items apply. 
\end{proof}

The coarse-grained distribution $\ cu\ $ of uniform distribution $\ u \ $  from Theorem \ref{cr1}-e, is given by
$\ \ \displaystyle (cu)_i=\frac{1}{|A||a|} \ \ $ for $\ i\in a. \ $ Thus the probability of a microstate is inversely proportional to the size of  its macrostate.  In the presence of large differences among the sizes of macrostates, small entropy microstates are assigned much larger probabilities that large entropy microstates.  \\

Given system $ \ (X,A,m,\alpha) \ \ $  set
$\ \ X^{\leq |c|}\ = \ \{i\in X \ | \ |i|\leq |c|\}. \ $ \\

\begin{thm}\label{ie}
{\em Let $ \ (X,A,f, \alpha,r) \ $ be a system 
such that $\ r \ $ preserves entropy, $\ c\in A^{\mathrm{neq}}\ $ be a non-equilibrium macrostate, and $\ q \in \mathrm{Prob}(A)\ $ be  given by 
$$ \displaystyle q_a= \frac{|a|}{|X^{\leq |c|}|} \ \ \ \mbox{if} \ \ \ |a|\leq |c|, \ 
\ \ \ \ \mbox{and} \ \ \ \ \ q_a=0 \ \ \ \mbox{if} \ \ \ |a|>|c|.
\ \ \ \ \ \mbox{We have that: } $$
\noindent a)  $ \displaystyle \mathrm{w}_{n}^ {q}(x) \geq 
  \mathrm{w}_{n}^{q}(-x) \ \ $ for $\ \ x\geq 0. $ \\

\noindent b) $ \displaystyle \sigma_{n}^q  \geq  0, \ $ and 
$\ \sigma_{n}^q=0  \ $ if and only if $\  X^{\leq |c|} \ $ is an $\ \alpha^n$-invariant set. \\
}
\end{thm}

\begin{proof}[\textbf{Proof}]

\noindent a)  Case $\ x=0\ $ is trivial. For $ \ x>0 \ $ 
using Theorem \ref{cr1}-a we get 

$$ \mathrm{w}_{n}^ {q}(x) \ \ = \ \ \frac{|\{i \in X^{\leq |c|}\ | \ \sigma_n(i)=x \}|}{|X^{\mathrm{\leq |c|}}|}\ \ = $$
$$ \frac{|\{i \in X^{\leq |c|} \ | \ \sigma_n(i)=x, \  r\alpha^ni \in X^{\leq |c|} \}|}{|X^{\leq |c|}|}
\ \ + \ \ \frac{|\{i \in X^{\leq |c|} \ | \ \sigma_n(i)=x, \  r\alpha^ni \in
 X^{> |c|}\}|}{|X^{\leq |c|}|}\ \ =  $$
$$ \frac{|\{i \in X^{\leq |c|} \ | \ \sigma_n(i)=-x \}|}{|X^{\leq |c|}|}
\ \ + \ \ \frac{|\{i \in X^{\leq |c|} \ | \ \sigma_n(i)=x, \  \alpha^ni \in
 X^{> |c|}\}|}{|X^{\leq |c|}|}\ \ =  $$
$$  \mathrm{w}_{n}^ {q}(-x)\ \ + \ \ \frac{|\{i \in X^{\leq |c|} \ | \ \sigma_n(i)=x, \  \alpha^ni \in
 X^{> |c|}\}|}{|X^{\leq |c|}|}\ \ \geq \ \
 \mathrm{w}_{n}^ {q}(-x).  $$
\noindent b) From item a) we get that
$$\sigma_{n}^q  \ = \ \sum_{x\in \mathbb{R}}x\mathrm{w}_{n}^ {q}(x) \ = \
\sum_{x\in \mathbb{R}_{> 0}}
x\big(\mathrm{w}_{n}^{q}(x)-\mathrm{w}_{n}^{q}(-x)\big) $$
$$\frac{1}{|X^{\leq |c|}|}\sum_{x\in \mathbb{R}_{> 0}}
x\ |\{i \in X^{\leq |c|} \ | \ \sigma_n(i)=x, \  \alpha^ni \in
 X^{> |c|}\}| \ \geq \ 0.$$ Thus $\ \sigma_{n}^q = 0 \ $ if and only if
$\  \{i \in X^{\leq |c|} \ | \ \sigma_n(i)=x, \ \alpha^ni \in
 X^{> |c|}\} = \emptyset \ \ $  for $\ x> 0, \ $ i.e. if and only if 
 $\ X^{\leq |c|} \ $ is an $\ \alpha^n$-invariant set. 
\end{proof}

The coarse-grained distribution $\ cq \ $ of the probability distribution  $\ q \ $ from
Theorem \ref{ie} is concentrated on microstates of entropy lesser or equal to
$\ \mathrm{ln}|c|, \ $ and it is uniform among such microstates.  \\

\begin{rem}
{\em Let $ \ (X,A,m,\alpha) \ $ have return time $\ d. \ $ From the Gibbs viewpoint \cite{g, xing} and in light of Theorem \ref{esl} it is natural to consider how
$$\ H(c[\alpha^n]q) \ = \ H([\alpha^n]q) + S([\alpha^n]q)\ \ \  \mbox{changes with} \ \  \ n, \ \ \ \ \mbox{for}\ \ \ q\in \mathrm{Prob}(A).$$
We have already shown that 
$\ \ H([\alpha^n]q)\ + \ S([\alpha^n]q)\ \geq \  H(q) + S(q) \ \ $ for all $\ \ n.\ \ $ 
Since $\ \alpha^d=1\ $ we have $\ \ H([\alpha^d]q)
 + S([\alpha^d]q) \ = \  H(q) + S(q). \ \ $ Therefore 
$$\mbox{either} \ \ \ \ H([\alpha^n]q) + S([\alpha^n]q) 
\ = \  H(q) + S(q) \ \ \ \  \mbox{for \ all} \ \ \ n, \ \  \ \ \mbox{or} $$
there is a smallest $\ 2 \leq k\leq d \ $ such that
$$\ \ H([\alpha^l]q) + S([\alpha^l]q)\ \geq \  
H([\alpha^{l-1}]q) + S([\alpha^{l-1}]q) \ \ \ \  \mbox{for} \ \ \  
\ 1 \leq l < k, $$ 
$$\mbox{and} \ \ \ \ H([\alpha^k]q) + S([\alpha^k]q) \ < \  
H([\alpha^{k-1}]q) + S([\alpha^{k-1}]q). $$

Here we limit ourselves to construct an example for which $\ k \ $  attains its minimal value $\ k=2. \ $ Consider a system with one block of size $\ n, \ $  one block of size $\ n^2, \ $ and dynamics $\ \alpha\ $ given diagrammatically, we display the case $\ n=3, \ $ where each row 
represents a block, and the arrow $\ \searrow\ $ means moving to the bottom of the next column: \\

$\ytableausetup{centertableaux, boxsize=0.7cm}
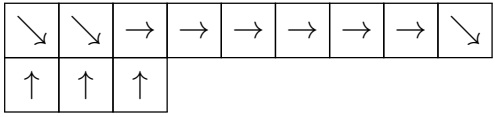
\begin{ytableau}
\searrow & \searrow & \rightarrow  & \rightarrow &  \rightarrow & \rightarrow & \rightarrow &  \rightarrow & \searrow\\
\uparrow & \uparrow & \uparrow 
 \end{ytableau}$ 

\

\

\noindent Let $\ q \ $ be  the probability distribution on blocks that assigns probability $\ 1 \ $ to the block with $\ n \ $ elements. We have $\ \ H([\alpha]q)= 0, \ $ 
$ \ S([\alpha]q)= 2\mathrm{ln}(n), \ $ and
$$\ \  H([\alpha^2]q)\ = \  -\frac{n-1}{n}\mathrm{ln}\big(\frac{n-1}{n}\big)
\ - \ \frac{1}{n}\mathrm{ln}\big(\frac{1}{n}\big), \ \ \ \ \ \
S([\alpha^2]q)\ = \ 
 \frac{n-1}{n}\mathrm{ln}(n)\ + \ \frac{2}{n}\mathrm{ln}(n). $$
For $\ n \geq 3, \ $ we have built a cyclic system with $\ n^2+n \ $ microstates, with a $\ \ \displaystyle \frac{1}{n+1}\ $ dominant equilibrium, and such that
$$H([\alpha]q)+ S([\alpha]q) - 
\big(H([\alpha^2]q)+S([\alpha^2]q)\big)\ = \ 
\frac{1}{n}\mathrm{ln}\bigg[\frac{(n-1)^{n-1}}{n}\bigg]
  \ > \ 0.$$
 }
\end{rem}

\section{Irreversibility from Attraction}\label{ifa}

In the previous section we made use of entropy preserving reversion maps, in this section we take a different approach where this condition no longer holds.
We begin showing that a reversible system, with  entropy preserving reversion, has equal decreasing and increasing entropy rates, thus 
$\ \varepsilon$-bound decreasing entropy rate is equivalent to entropy being mostly constant. If $\ |D|< |I|\ $
in a reversible system,  then the reversion map 
$\ r \ $ does not preserve entropy. \\

\begin{prop}
{\em Consider system $\ (X,A,m, \alpha,r)\ $ with $\ r \ $ entropy preserving.\\

\noindent a)  $\ |D| =  |I|. \ \ \ $ \ \ 
b) $\ \displaystyle \frac{|D|}{|X|} \leq  \varepsilon \ \ $ if and only if
$\ \ \displaystyle \frac{|C|}{|X|} \geq 1- 2\varepsilon .\ \ $ \ \ 
c) For $\ n\geq 1 \ $ set 
$$D_n=\{i\in X  \ | \ S(i) > S(\alpha i) > \cdots > S(\alpha^{n-1} i) > S(\alpha^n i) \},$$
$$I_n=\{i\in X  \ | \ S(i) < S(\alpha i) < \cdots < S(\alpha^{n-1} i) < S(\alpha^n i) \}.$$
We have that  $\ \ |D_n|=|I_n| .$

 }
\end{prop}

\begin{proof}[\textbf{Proof}]
Item b) follows from item a) and (\ref{dic=1}). \
Item a) is case $\ n=1\ $ of item c). 
We show that the map  $\ r\alpha^n\ $ restricts to a bijection
$$\ r\alpha^n: D_n \rightleftarrows I_n : r\alpha^n. \ $$ 

\noindent  Let $\ \ i \in D_n, \ \ $ since $\ r\ $ preserves entropy we have that $$ \ S(ri) \ > \ S(r\alpha i)\ > \ \cdots \ > \ S(r\alpha^{n-1} i)\ > \ S(r\alpha^n i).  $$

\noindent Therefore  $\ r\alpha^n i \in I_n\ $ since we have  
$\ \alpha$-path
$$r\alpha^n i \ \ \to \ \ r\alpha^{n-1} i \ \ \to \ \ \cdots \ \ \to
\ \ r\alpha i \ \ \to \ \ ri.$$

\noindent A similar argument shows that $\ \ r\alpha^n i \in D_n\ \ $
if $\ \ i\in I_n. \ $
\end{proof}

\

Consider system $\ (X,\alpha)\ $ and $\ E\subseteq X.\ $ We say that the system is  $\ E$-\textbf{bound} if each $\ \alpha$-cycle intersects $\ E \subseteq X. \ $ For such a system, the $\ E$-\textbf{reaching time map} $\ e: X \to \mathbb{N} \ $ is given by
\begin{equation}\label{eq1}
e(i) \ = \  \mbox{smallest} \ \ k\geq 0 \  \ \mbox{such\ that \ } \alpha^k(i) \in E. 
\end{equation}
We let $ \ L = \underset{i\in X}{\mathrm{max}}\ e(i)\ $ be the \textbf{largest} $\ E$-\textbf{reaching time}. \\

Given system $\ (X,\alpha)\ $ we say that $\ E \subseteq X \ $ is  $\ S$-\textbf{stable} if for any microstate $\ i\in E\cap\alpha(X\setminus E)\ $
we have that $\ \alpha^{k}i \in E \ $ for $\ 0 \leq k \leq S. \ $
So  $\ (X,A,m, \alpha)\ $ has a $\ S$-\textbf{stable equilibrium} 
if whenever a microstate arrives to the equilibrium it remains in the equilibrium for at least $\ S \ $ units of time.  \\ 

Let  $\ A: X \to  \mathbb{N} \ $ be given by demanding that
entropy defines an \textbf{arrow of time} of length $\ A(i) \ $ at microstate $\ i, \ $ in the sense that $\ A(i)\ $  is the largest integer such that
$\ \ S(i) \ < \ S(\alpha i) \ < \cdots \ < \ S(\alpha^A i)  , \ \ $
i.e. beginning at $\ i \ $ entropy is strictly increasing for exactly $\ A(i) \ $ units of time. Thus, if the microstates along the $\ \alpha$-path of length $\ A(i) \ $ starting at $\ i \ $ are given in an arbitrary unordered fashion, then the $\ \alpha$-path ordering can be recovered computing $\ S \ $
instead of $\ \alpha. \ $\\

\begin{thm}\label{g0}
{\em Let $\ (X,\alpha)\ $ be $\ E$-bound.
The system $ \ (X,[0,L],e,\alpha) \ $ satisfies:\\

\noindent a) $E\ $ is an equilibrium macrostate. Entropy always increases  on $\ X^{\mathrm{neq}}. \ $  If $\ E\ $ is $\ S$-stable, then $\ X^{\mathrm{eq}}\ $ is $\ S$-stable. Set $\ X_k= e^{-1}(k) \ $ for $\ k \in [0,L],\ $ then 
\begin{equation}\label{z1}
|\llbracket E,X_k\rrbracket| \ = \ |X_k|-|X_{k+1}|
\ \ \ \ \ \ \mbox{and} \ \ \ \ \ \ 
|X_k| \ = \ \sum_{s=k}^{L}|\llbracket E,X_s\rrbracket|.
\end{equation}
\begin{equation}\label{w222}
\mbox{Furthermore} \ \ \ T_{ik}\ = \ \delta_{i-1,k} \ \  \ \mbox{for} \ \ \ i\geq 1,\ \ \ 
\mbox{and} \ \ \   \displaystyle T_{0k}\ = \ \frac{|X_k|-|X_{k+1}|}{|E|}.  
\end{equation}

Assume that $\ e\alpha E=[0,L]. \ $\\

\noindent b) $ X^{\mathrm{eq}}=E,\ $ 
$ \ X^{\mathrm{neq}}=X\setminus E,\  $ $\ D= \{i \in X^{\mathrm{eq}} \ |\ \alpha i \in X^{\mathrm{neq}} \}, \ $ entropy is strictly
increasing on $\ X^{\mathrm{neq}}.\  $
Furthermore, $\ \ |D|=|X_1|,\ \ |I|=| X^{\mathrm{neq}}|,\ \ 
|C|=|X^{\mathrm{eq}}|-|X_1|.$   \\

\noindent c) The mean value $\ \langle A \rangle\ $  of  $\ A \ $ is given by
\begin{equation}\label{zz-1}
\langle A \rangle  \ =  \ \frac{1}{|X|}\sum_{k=1}^{L}k|X_k|
\ = \ \frac{|X^{\mathrm{eq}}|}{|X|}\sum_{k=1}^{L}\binom{k+1}{2}[\alpha]_{0,k}.  
\end{equation}
}
\end{thm}

\begin{proof}[\textbf{Proof}]
 a) Note that if $\ i \in X_k,\ $ with $\ k>0, \ $ then 
 $\ e(\alpha i)= e(i)-1, \ $ and thus $\ \alpha i \in X_{k-1}.\ $ 
Take $\ i \in X\ $ such that $\ e(i)=L, \ $ then 
 $\ e(\alpha^k i)=L-k \ $ for $\ k\in [0,L],\ $ so $\ e \ $ 
is surjective. We have maps $\ X_L  \xrightarrow{\alpha} X_{L-1}  \xrightarrow{\alpha} 
\cdots X_1 \xrightarrow{\alpha} X_0, \ $ thus
\begin{equation}\label{t56}
 |X_L| \ \leq \  |X_{L-1}| \ \leq \  
\cdots \ \leq \ |X_1| \ \leq \ |X_0|
\end{equation}
since $\ \alpha \ $ is bijective. 
Note that $\ X_0=E.\ $
Setting $\ X_{L+1}=\emptyset,\ $ we get 
\begin{equation}\label{t57}
\ \   X_k\ = \ \alpha X_{k+1}  \  \sqcup \ 
\big(X_k \cap \alpha E\big) \ \ \ \ \mbox{for} \ \ \ \
0\leq k\leq L.
\end{equation}
Thus  $\ |X_{k+1}| < |X_{k}|\ $ if and only if 
$\ X_k \cap \alpha E\neq \emptyset .\ $ Identities (\ref{z1}) follow by recursion from (\ref{t57}) since
\begin{equation}\label{zz1}
|X_k|\ = \ |\alpha X_{k+1}|  \  + \ 
|X_k \cap \alpha E\big| \ = \ 
 |X_{k+1}|  \  + \ |\llbracket E,X_k\rrbracket |.
\end{equation}
From (\ref{t56}) we get that $\  T_{ik}=\delta_{i-1,k} \ $ for $\  i\geq 1. \ $ Identity $\  \displaystyle [\alpha]_{0k}= \frac{|X_k|-|X_{k+1}|}{|E|}\ $  follows from (\ref{zz1}).\\

Assume that $\ E \ $ is $\ S$-stable. If $\ i \in X^{\mathrm{eq}}, \ $ then $\ i \in X_k \ $ where
$\ |X_k|=|E|. \ $  Therefore $\ \alpha^{u}(i) \in X_{k-u} \subseteq X^{\mathrm{eq}} \ $
for $\ 0 \leq u \leq k-1, \ $ and  $\ \alpha^{u}(i) \in X_0=E \subseteq X^{\mathrm{eq}} \ $
for $\ k \leq u \leq S+k. \ $ Thus $\ X^{\mathrm{eq}} \ $ is
$\ S$-stable.\\

For the remainder of the proof we assume that $\ e\alpha E=[0,L]. \ $ \\

\noindent b) By item a) we have 
$ \ X^{\mathrm{eq}} = X_0=E\ $ and $\  X^{\mathrm{neq}} = \coprod_{k=1}^{L}X_k. \  $  Entropy is strictly increasing on $\ X^{\mathrm{neq}}, \ $ indeed if 
$ \ i \in X_{k+1}\ $ then 
$$ \ S(i) \ = \ \mathrm{ln}|X_{k+1}| \ < \ \mathrm{ln}|X_k|\ = \ S(\alpha i). \ 
\ \ \ \ \ \ \mbox{We have that}$$ 
$$ DX \  = \ \{i \in X \ |\ e(\alpha i)>e(i) \} \ = \
\{i \in X^{\mathrm{eq}} \ |\ e(\alpha i)>0\} \ = \
\{i \in X^{\mathrm{eq}} \ |\ \alpha i \in X^{\mathrm{neq}} \}.
$$ 
From (\ref{z1}) we get that
$$|D|\ = \ |\{i \in X^{\mathrm{eq}} \ |\ \alpha i \in X^{\mathrm{neq}} \}| \ = \
\sum_{k=1}^{L}|\llbracket X^{\mathrm{eq}},X_k\rrbracket| \ = \
 \sum_{k=1}^{L}|X_k|-|X_{k+1}| \ = \ |X_1|.$$
Also $\ \  |I| =  |X^{\mathrm{neq}}|,\ \ $
and $\ \ |C|=|X^{\mathrm{eq}}|-|D| \ = \ |X^{\mathrm{eq}}|-|X_1|.$ \\

\noindent c) The key observation is that 
$\ A  =    e  \ $
by (\ref{w222}). This immediately implies the first identity in
(\ref{zz-1}),  and then the second identity follows from (\ref{z1}), indeed
$$\langle A \rangle  \ =  \ \frac{1}{|X|}\sum_{k=1}^{L}k|X_k|
\ = \  \ \frac{1}{|X|}\sum_{k=1}^{L}k\bigg(\sum_{s=k}^{L}|\llbracket X^{\mathrm{eq}},X_s\rrbracket|\bigg)
\ =  $$
$$  \frac{|X^{\mathrm{eq}}|}{|X|}\sum_{k=1}^{L}\sum_{s=k}^{L}k[\alpha]_{0,s}
\ = \ \frac{|X^{\mathrm{eq}}|}{|X|}
\sum_{s=1}^{L}\big(\sum_{k=1}^{s}k\big)[\alpha]_{0,s}
\ = \
\frac{|X^{\mathrm{eq}}|}{|X|}\sum_{s=1}^{L}\binom{s+1}{2}[\alpha]_{0,s}.$$

\end{proof}

\begin{rem}
{\em If we think of $\ E \ $ as the equilibrium, the condition $\ e\alpha E=[0,L] \ $ means that if there is a perturbation of the equilibrium of order $\ L\ $ 
( a path of length $\ L+1\ $ from the equilibrium to the equilibrium via non-equilibrium points, ) then there are perturbations of the equilibrium for all smaller orders. Indeed it is reasonable to expect that shorter perturbations are
more likely than longer perturbations. 
  }
\end{rem}

\begin{thm}\label{g8}
{\em Let $\ (X,\alpha,r)\ $ be a $\ E$-bound reversible system such that
$\ E\subseteq X\ $ is $\ r$-invariant.  Consider the reversible system $ \ (X,[0,L],e,\alpha,r). \ $ \\

\noindent a) Let $\ i_{k+1}\to i_k \to \cdots \to i_1 \to i_0\ $  be an $\ \alpha$-path with $\ i_0, i_{k+1} \in E \ $ and 
$\ i_u \in X\setminus E \ $  for $\ u \in [1,k].\ $
Then $\ ri_{0}\to ri_1 \to \cdots \to ri_k \to ri_{k+1}\ $ is also 
an $\ \alpha$-path,  $\ ri_{0}, ri_{k+1} \in E \ $ and 
$\ ri_u \in X\setminus E \ \ $  for $\ \ u \in [1,k].\ $ \\

\noindent b)  $ S(ri_u) = S(i_u)\ \ \  \mbox{if and only if} \ \ \  |X_u|  =  |X_{k-u+1}|\ \ $ for $\ \ u \in [1,k].\ $\\

\noindent c) If $\ e\alpha X^{\mathrm{eq}}=[0,L], \ $ then
$ \displaystyle S(ri_u) = S(i_u)\ \ \   \mbox{if and only if} \ \ \ 
 u  = \frac{k+1}{2}, \ $ i.e. if and only if $\ i_u\ $
is equidistant to $\ X^{\mathrm{eq}} \ $ in the forward and the backward directions.
}
\end{thm}

\begin{proof}[\textbf{Proof}]
 a) $ E\ $ and $\ X\setminus E\ $ are $\ r$-invariant, and we have
$\ \alpha r i_u\ =\ r\alpha^{-1}i_u\ = \ ri_{u+1}\ $
for $\ u\in [0,k]. \  $ \\

\noindent b) By item a) it follows that $\ e(i_u)=u \ $ and $\ e(ri_u)=k-u+1, \ $ and thus 
$S(i_u)= \mathrm{ln}|X_u| \ $ and $\ S(ri_u)= \mathrm{ln}|X_{k-u+1}|\ $ for $\ u \in [1,k]. \ $\\

\noindent c) Assume $\ e\alpha E=[0,L], \ $ then
$\ |X_u|  =  |X_{k-u+1}| \ $ if and only if $\ u  =  k-u+1,\ $ i.e.
if and only if $\ \displaystyle  u  = \frac{k+1}{2}. \ $ The result follows from
item b).
\end{proof}

Given $\ E$-bound system $\ (X,\alpha)\ $ we enhance it
to an  $\ \widehat{E}$-bound reversible system 
$\ (\widehat{X},\widehat{\alpha},r)\ $ where 
$\ \widehat{X} =   X\times \{-1,1\}, \ $ $\ \widehat{E}  =   E\times \{1\} \sqcup E\times \{-1\},\ $ the map
$\ \widehat{\alpha}:\widehat{X} \to \widehat{X} \ $ is given by
$\ \alpha(i,s)\ = \ (\alpha^{s}i,s), \ $ 
and the reversion map $\ r: \widehat{X} \to \widehat{X} \ $ is given
by $\ r(i,s) \ = \ (i,-s). \ $ We check that $\ r\widehat{\alpha} r = \widehat{\alpha}^{-1}: \ $ 
\begin{equation}\label{z100}
r\widehat{\alpha} r(i,s) \ = \ r\widehat{\alpha}(i,-s)
\ = \ r(\alpha^{-s}i,-s) \ = \ (\alpha^{-s}i,s) \ = \ 
\widehat{\alpha}^{-1}(i,s).
\end{equation}
Next we check that the new system is $\ \widehat{E}$-bound. Note that
$\ \widehat{\alpha}^{k}(i,s)\ = (\alpha^{sk}(i),s) \ $ belongs to $\ \widehat{E} \ $ if and only if $\ \alpha^{sk}(i) \ $ belong to $\ E. \ $ We can always find $\ k\geq 0\ $ fulfilling the latter condition since $\ (X,\alpha)\ $ is $\ E$-bound.\\

Applying Theorem \ref{g8} we get reversible system 
$\ ( \widehat{X},  [0,\widehat{L}],  \widehat{e}, \widehat{\alpha},  r  ) \ $  where
$$\widehat{e}(i,s)\ = \ \mbox{smallest} \ \ k\geq 0 \  \ \mbox{such\ that \ } \ \widehat{\alpha}^k(i,s) \in \widehat{E}, $$
$$\mbox{and} \ \ \ \ \ \  \widehat{L} \ = \ \underset{(i,s)\in \widehat{X}}{\mathrm{max}}\ \widehat{e}(i,s). \ \ \ \ \ \  \ 
  \ \ \ \ \ \  \ \   $$

\begin{thm}\label{bb1}
{\em 
System $ \ ( \widehat{X}, [0,\widehat{L}],  \widehat{e}, \widehat{\alpha},  r   ) \ $ has the following properties:\\

\noindent a) $\widehat{E} \ $ is $\ r$-invariant.   $\ \ \widehat{E}\ $ is $\ S$-stable if and only if $\ E \ $ is $\ S$-stable.
We have  $$\ \  \widehat{L} \ = \ \underset{(i,s)\in \widehat{X}}{\mathrm{max}}\ \widehat{e}(i,s)
  \ = \ \underset{i\in X}{\mathrm{max}}\ e(i)
  \ = \ \underset{i\in X}{\mathrm{max}}\ e_{-1}(i)\ = \ L,\ \ $$
  where $\ e_{-1}(i)\ $ is the smallest $\ k\geq 0 \ $ such that
$\ \alpha^{-k}(i)\in E.  \ $ 
We have that $ \ \widehat{X}_k \ = \ X_k \sqcup  X_{-k}, \ $ where $\ X_{-k} = \{i\in X \ | \ e_{-1}(i)=k \} \ $ and  $\ |\widehat{X}_k|=2|X_k|. \ $\\

\noindent b)  Let $\ \ (i_{k+1},s)\to (i_k,s) \to \cdots \to (i_1,s) \to (i_0,s) \ \ $  be an $\ \widehat{\alpha}$-path with microstates $\ (i_0,s), (i_{k+1},s) \ $ in $\ \widehat{E}, \ \ $ and 
$\ \ (i_u,s) \in \widehat{X}\setminus \widehat{E} \ \ $  for $\ \ u \in [1,k].\ \ $  We have that
$$ \ S(r(i_u,s))=S(i_u,s) \ \ \ \ \mbox{if and only if} \ \ \ \
|X_u|  =  |X_{k-u+1}|.$$

\noindent c) $e\alpha E=[0,L]\ \ \ $ if and only  $\ \ \ \widehat{e}\widehat{\alpha} \widehat{E}=[0,L]. \ $ \\

Assume that $\ e\alpha E=[0,L]. \ $ \\

\noindent d) $\widehat{X}^{\mathrm{eq}}=\widehat{E},\  $  
$ \widehat{X}^{\mathrm{neq}}=\widehat{X}\setminus \widehat{E},\  $ 
$\ \widehat{D}= \{i \in \widehat{X}^{\mathrm{eq}} \ |\ \alpha i \in \widehat{X}^{\mathrm{neq}} \}, \  $ entropy is strictly increasing on $\ \widehat{X}^{\mathrm{neq}},\ $
and 
$\ \ \ |[\widehat{X}^{\mathrm{eq}},\widehat{X}_k]| \ = \ 2(|X_k|-|X_{k+1}|)\ \ \  \mbox{for} \ \ \  k\in[0,L].\ $ 
 Also
$ \ S(r(i_u,s))=S(i_u,s) \ \ $ if and only if 
$\ \ \displaystyle u= \frac{k+1}{2}. \ \ $ Furthermore
$$\ \ \ \displaystyle \frac{|\widehat{D}|}{|\widehat{X}|}=\frac{|D|}{|X|}, \ \ \ 
\frac{|\widehat{I}|}{|\widehat{X}|}=\frac{|I|}{|X|}, \ \ \ 
\frac{|\widehat{C}|}{|\widehat{X}|}=\frac{|C|}{|X|}, \ \ \ \
\mbox{and} \ \ \ \ \langle \widehat{A} \rangle  =  \langle A \rangle. \ \  $$
}  
\end{thm}

\begin{proof}[\textbf{Proof}] 
a) $\widehat{E}\ $ is $\ r$-invariant since 
$\ r(E\times \{s\}) =  E\times \{-s\} \ $ for $\ s\in \{-1,1\}.\ \ $
Looking at microstates $\ (i,1) \ $
we see that if $\ \widehat{E}\ $ is $\ S$-stable then $\ E \ $ is $\ S$-stable.
Conversely if $\ E \ $ is $\ S$-stable, then $\ E \ $ is also $\ S$-stable
with dynamics given by $\ \alpha^{-1},\ $ thus $\widehat{E}\ $ is $\ S$-stable.\ \  Note that $\ \widehat{e}(i,1)=e(i) \ $ and $\ \widehat{e}(i,-1)=e_{-1}(i). \ $ \ Furthermore  
$\ L\ =\ \underset{i\in X}{\mathrm{max}}\ e(i) \ = \ \underset{i\in X}{\mathrm{max}}\ e_{-1}(i), \ $ indeed for
$\ i_0, i_{k+1} \in E \ $ and  $\ i_s\in  X\setminus E \ $  for $\ s \in [1,k],\ $ we have $\ \alpha$-path
$\ i_0 \to i_1 \to \cdots \to i_k \to i_{k+1}\ $ if and only if we have  $\ \alpha^{-1}$-path $\ i_{k+1}\to i_k \to \cdots \to i_1 \to i_0. \ $
Therefore
$$\underset{(i,s)\in X}{\mathrm{max}}\ \widehat{e}(i,s) \ = \
\mathrm{max}\bigg\{\ \underset{i\in X}{\mathrm{max}}\ e(i) \ , \ \underset{i\in X}{\mathrm{max}}\ e_{-1}(i)\ \bigg\}
\ = \ \mathrm{max}\{ L , L \} \ = \ L.\ $$ 
We have that $\ (i,s)\in \widehat{X}_k\ $ if and only if either $\ s=1\ $
and $\ e(i)=k,\ $ or $\ s=-1\ $
and $\ e_{-1}(i)=k,\ $ thus $\ \widehat{X}_k = X_{k} \sqcup X_{-k}. \ $
Clearly $\ X_{0} = X_{-0} = E. \ $
Note that
$$X_{1}  =  \coprod_{u=1}^{L}\alpha^{u-1}(X_u\cap\alpha E), \ \ \
X_{2}  =  \coprod_{u=2}^{L}\alpha^{u-2}(X_u\cap\alpha E), \ \ \
X_{3}  =  \coprod_{u=3}^{L}\alpha^{u-3}(X_u\cap\alpha E), \ \ \ \cdots $$ 
$$X_{-1}  =  \coprod_{u=1}^{L}(X_u\cap\alpha E), \ \ \
X_{-2}  =  \coprod_{u=2}^{L}\alpha(X_u\cap\alpha E), \ \ \
X_{-3}  =  \coprod_{u=3}^{L}\alpha^{2}(X_u\cap\alpha E), \ \ \ \cdots $$ 

One can visualize this identities with the help of a diagram,
below we display the case $\ L=5. \ $
For $\ 1\leq j \leq L\ $ and $\ 1 \leq i \leq L-j+1, \ $ the block in  row $\ i, \ $
column $\ j ,\ $ represents the set microstates  $\ \ \alpha^{L-i-j+1}(X_{L-j+1}\cap\alpha E), \ \ $ while the entries 
$\ s,-t \ $ tell that microstates in this block 
need $\ s\ $ applications of $\ \alpha, \ $ and $\ t\ $ applications of 
$\ \alpha^{-1},\ $ to enter the equilibrium, respectively. \\

$\ytableausetup{centertableaux, boxsize=1.2cm}
\begin{ytableau}
1,-5 & 1,-4 & 1,-3    & 1,-2 & 1,-1\\
2,-4 & 2,-3 & 2,-2  & 2,-1  \\
3,-3 & 3,-2 & 3,-1    \\
4,-2 & 4,-1  \\
5,-1 
\end{ytableau}$

\

\

\noindent In general for $\ k \in [1,L] \ $ we have that 
$$X_{k}  =  \coprod_{u=k}^{L}\alpha^{u-k}(X_u\cap\alpha E)
\ \ \ \ \ \mbox{and} \ \ \ \ \ 
X_{-k}  =  \coprod_{u=k}^{L}\alpha^{k-1}(X_u\cap\alpha E).$$  \\
Since $\ |\alpha^{k-1}(X_u\cap\alpha E)| \ = \  |X_u\cap\alpha E| \ = \
| \alpha^{u-k}(X_u\cap\alpha E)|, \ $  we get
$\ |X_{-k}| = |X_{k}|,\ $ and therefore 
$\ |\widehat{X}_{k}|\ = \ |X_k|\ + \ |X_{-k}|\ = \ 2|X_k|.$\\

\noindent b) Theorem \ref{g8}-b  gives $ \ S(r(i,s))=S(i,s) \ \ \  \mbox{if and only if} \ \ \  |\widehat{X}_s|  =  |\widehat{X}_{k-s+1}|. \ \ $ The result follows by  item a) since $\ |\widehat{X}_{s}| = 2|X_s|\ $ and
$\ |\widehat{X}_{k-s+1}| =  2|X_{k-s+1}|.\ $  \\

\noindent c) Assume  $\ e\alpha E=[0,L]m \ $ and  take $\ k \in [0,L]. \ $ There is
$\ i\in E\ $ such that $\ e(\alpha i)=k,\ $ then
$\ \widehat{e}\widehat{\alpha}(i,1)=k ,\ $ 
and thus $\  \widehat{e}\widehat{\alpha} \widehat{E}=[0,L]. \ $ 
Conversely, assume that  $\  \widehat{e}\widehat{\alpha} \widehat{E}=[0,L] \ $
and take $\ k\in [0,l], \ $ then either there is $\ i \in E \ $ such
that $\ e(\alpha i)=\widehat{e}\widehat{\alpha}(i,1)=k\ $ and thus
$\ k\in  e\alpha E,\ $ or there is $\ i \in E \ $ such
that $\ \widehat{e}\widehat{\alpha}(i,-1)=k,\ $ thus we have $\ \widehat{\alpha}$-path
$\ (i_0,-1) \to (i_1,-1) \to \cdots \to (i_k,-1) \to (i_{k+1},-1) \ $
where $\ i_0,i_{k+1} \in E\ $ and $\ i_u \notin E\ $ for $\ u\in [1,k], \ $
and thus the $\ \alpha$-path 
$\ i_{k+1} \to i_k \to \cdots \to i_1 \to i_0 \ $
shows that $\ e(\alpha i_{k+1}) = k, \ $ therefore we conclude
that $\ e\alpha E=[0,L].\ $ \\

\noindent d)  $ \widehat{e}\widehat{\alpha} \widehat{E}=[0,L] \ $ since
$\ e\alpha E=[0,L].\ $  The desired properties follow directly from item a), Theorem \ref{g0}-b)-c),
and Theorem \ref{g8}-c).
\end{proof}

\

We left for future work the study of weak versions of $\ E$-bound systems that allow a small number of cycles to avoid $\ E, \ $ as well as weak versions of $\ S$-stability that allow a small number of microstates to enter and leave the equilibrium in a short interval of time.

\

\noindent ragadiaz@gmail.com \\
\noindent Departamento de Matem\'aticas\\
\noindent  Universidad Nacional de Colombia - Sede Medell\'in, Colombia

\end{document}